\documentclass[10pt, a4paper]{article}                                                                            
\usepackage{amsmath,amsfonts,amsthm,amssymb}
\usepackage{amsrefs}                           

\usepackage{geometry}                        
\usepackage{graphicx}                          
\usepackage{setspace}  
\usepackage{caption}
\usepackage{listings}
\usepackage{tabularx}
\usepackage{xcolor}
\usepackage{indentfirst}
\usepackage{cite}
\usepackage{subfigure}

\usepackage{multirow}

\usepackage{verbatim}

\usepackage{enumerate}

\usepackage[title]{appendix}

\usepackage[title]{appendix}

\usepackage{bm}

\usepackage[font=footnotesize,labelfont=bf]{caption}

\makeatletter
\@addtoreset{equation}{section}
\makeatother

\newtheorem{proposition}{Proposition}[section]
\newtheorem{corollary}{Corollary}[section]

\newtheorem*{rem}{Remark}

\newcommand\dd{\mathrm{d}}
\newcommand\n{\mathbf{n}}
\newcommand\e{\mathbf{e}}
\newcommand\m{\mathbf{m}}
\newcommand\p{\mathbf{p}}
\newcommand\vvec{\mathbf{v}}
\newcommand\w{\mathbf{w}}
\newcommand\x{\mathbf{x}}
\newcommand\Qvec{\mathbf{Q}}
\newcommand\E{\mathbf{E}}
\newcommand\pp{\partial}
\newcommand\tr{\mathrm{tr}}

\title{Remarks on Uniaxial Solutions in the Landau-de Gennes Theory }
\author{Apala Majumdar$^1$\thanks{Corresponding author. Email: \tt{A.Majumdar@bath.ac.uk}} ~and Yiwei Wang$^2$ \vspace{1em} \\ 
       \small{$^1$ Department of Mathematical Sciences, University of Bath, Bath BA2 7AY, UK.} \\ 
       \small{$^2$ School of Mathematical Sciences, Peking University, Beijing 100871, China.} \\}
       
\date{}

\begin{document}
\maketitle                            

\begin{abstract}
We study uniaxial solutions of the Euler-Lagrange equations for a Landau-de Gennes free energy for nematic liquid crystals, with a fourth order bulk potential, with and without elastic anisotropy. In the elastic isotropic case, we show that (i) all uniaxial solutions of the Euler-Lagrange equations, with a director field of certain symmetry, necessarily have the radial-hedgehog structure modulo an orthogonal transformation, (ii) the``escape into third dimension'' director cannot correspond to a purely uniaxial solution of the Landau-de Gennes Euler-Lagrange equations and we do not use artificial assumptions on the scalar order parameter and (iii) there are no non-trivial uniaxial solutions that have $\mathbf{e}_z$ as an eigenvector. In the elastic anisotropic case, we prove that all uniaxial solutions of the corresponding Euler-Lagrange equations, with a certain symmetry, are strictly of the radial-hedgehog type, i.e. the elastic anisotropic case enforces the radial-hedgehog structure (or the degree $+1$-vortex structure) more strongly than the elastic isotropic case and the associated partial differential equations are technically far more difficult than in the elastic isotropic case.
\end{abstract}



\section{Introduction}
\label{sec:intro}
Nematic liquid crystals are classical examples of mesophases intermediate in physical character between conventional solids and liquids \cite{dg, newtonmottram}. Nematics are often viewed as complex liquids with long-range orientational order or distinguished directions of preferred molecular alignment, referred to as directors in the literature. The orientational anisotropy of nematics makes them the working material of choice for a range of optical devices, notably they form the backbone of the multi-billion dollar liquid crystal display industry.

Continuum theories for nematics are well-established in the literature and we work within the powerful Landau-de Gennes (LdG) theory for nematic liquid crystals. The LdG theory describes the nematic phase by a macroscopic order parameter, the $\mathbf{Q}$-tensor order parameter that describes the orientational anisotropy in terms of the preferred directions of alignment and ``scalar order parameters'' that measure the degree of order about these directions. Mathematically, the  $\Qvec$-tensor is a symmetric, traceless
$3\times3$ matrix, with five degrees of freedom \cite{dg, newtonmottram}. A nematic phase is said to be (i)~\textit{isotropic} if $\Qvec=0$, (ii)~\textit{uniaxial} if $\Qvec$ has two degenerate non-zero eigenvalues with a single distinguished eigenvector and (iii)~\textit{biaxial}
if $\Qvec$ has three distinct eigenvalues.
In particular, if $\Qvec$ is uniaxial, it can be written in the form
\begin{equation}
\label{eq:uniaxial}
\Qvec = s\left( \n \otimes \n - \frac{\mathbf{I}}{3} \right), 
\end{equation}
where $\n$ is the distinguished eigenvector with the non-degenerate eigenvalue, labelled as the ``uniaxial'' director, $s$ is a scalar order parameter that measures the degree of order about $\n$, and $\mathbf{I}$ is the $3 \times 3$ identity matrix \cite{majumdarcpam2012}. The eigenvalues of the uniaxial $\Qvec$ are $\frac{2s}{3}, -\frac{s}{3}, -\frac{s}{3}$ respectively and $s=0$ describes a locally isotropic point. The uniaxial $\Qvec$-tensor only has three degrees of freedom and the mathematical analysis of uniaxial $\Qvec$-tensors has strong analogies with Ginzburg-Landau theory, since we can treat uniaxial $\Qvec$-tensors as $\mathbb{R}^3 \to \mathbb{R}^3$ maps \cite{majumdarcpam2012}.

As with most variational theories in materials science, the experimentally observed equilibria are modelled by either global or local minimizers of a LdG energy functional \cite{dg, newtonmottram, amaz}. The LdG energy typically comprises an elastic energy and a bulk potential; the elastic energy penalizes spatial inhomogeneities and the bulk potential dictates the isotropic-nematic phase transition as a function of the temperature \cite{newtonmottram, amaz}. There are several forms of the elastic energy; the Dirichlet energy is referred to as the ``isotropic'' or ``one-constant'' elastic energy and elastic energies with multiple elastic constants are labelled as ``anisotropic'' in the sense that they have different energetic penalties for different characteristic deformations \cite{gartlanddavis}. These equilibria are classical solutions of the associated Euler-Lagrange equations, which are a system of five elliptic, non-linear partial differential equations for reasonable choices of the elastic constants \cite{gartlanddavis}.  This paper focuses on necessary and sufficient conditions for the existence/non-existence of purely uniaxial solutions for the LdG Euler-Lagrange equations, with and without elastic anisotropy. This is a highly non-trivial analytic question; uniaxial $\mathbf{Q}$-tensors only have three degrees of freedom and to date, there are few explicit examples of uniaxial solutions for this highly coupled system.

Our computations build on the results  in \cite{majumdarcpam2012} and \cite{lamy2015uniaxial}, although both papers focus on the elastic isotropic case. In the paper \cite{majumdarcpam2012}, the author derives the governing partial differential equations for the order parameter $s$ and three-dimensional director field, $\mathbf{n}$ in (\ref{eq:uniaxial}) in the one-constant LdG case and studies uniaxial minimizers (if they exist) of the corresponding energy functional in a certain asymptotic limit. In \cite{lamy2015uniaxial}, the author addresses some general questions about the existence of uniaxial solutions for the one-constant LdG Euler-Lagrange equations. The author derives an ``extra equation'' that needs to be satisfied by the director in ``non-isotropic'' regions; this equation heavily constrains uniaxial equilibria. The author further shows that if the uniaxial solution is invariant in a given direction, then the uniaxial director is necessarily constant in every connected component of the domain; we refer to such uniaxial solutions as ``trivial'' uniaxial solutions. In \cite{lamy2015uniaxial}, the author
proves that for the model problem of a spherical droplet with radial boundary conditions, the ``radial-hedgehog'' solution is the unique uniaxial equilibrium for all temperatures, for a one-constant elastic energy density. The radial-hedgehog solution is analogous to the degree $+1$ vortex in the Ginzburg-Landau theory for superconductivity \cite{brezisbethuelhelein}; the director field $\mathbf{n}$ is simply the radial unit-vector in three dimensions and the scalar order parameter, $s$, is a solution of a second-order nonlinear ordinary differential equation which vanishes at the origin (see (\ref{eq:uniaxial})). It is not yet clear if there are other explicit uniaxial solutions of the Euler-Lagrange equations, even in the one-constant case, in three dimensions.

We re-visit the question of purely uniaxial solutions for the LdG Euler-Lagrange equations, in the elastic isotropic and anisotropic cases, without the restriction of special geometries or specific boundary conditions.
Whilst we do not provide a definitive answer to the question - are there other non-trivial uniaxial solutions, apart from the well-known radial-hedgehog solution, for the fully three-dimensional (3D) Euler-Lagrange equations; we make progress by considering special cases and excluding the existence of other non-trivial uniaxial solutions for these special cases. Our main results can be summarized as follows. We firstly characterize the uniaxial solution in terms of the scalar order parameter, $s$, and two angular variables, $f$ and $g$, that parameterize the uniaxial director. We derive the five governing partial differential equations for these three variables from the one-constant Euler-Lagrange equations and in particular, we recast the ``extra condition'' in \cite{lamy2015uniaxial} in terms of $f$ and $g$. This is an interesting and useful computation that has not been previously reported in the literature. In terms of spherical polar coordinates, $\left(r, \varphi, \theta \right)$ where $r$ is the radial distance in three dimensions, $0\leq \varphi\leq \pi$ is the polar angle and $0\leq \theta < 2\pi$ is the azimuthal angle, the radial-hedgehog solution corresponds to $f = \varphi$ and $ g = \theta$ with $s$ being a solution of a second-order ordinary differential equation. We prove that for a separable director field with $f = f(\varphi)$ or $g = g(\theta)$, all admissible uniaxial solutions must have $f=\pm \varphi$, $g=\pm \theta + C$ for a real constant $C$ and $s$ is a solution of the ``radial-hedgehog'' ordinary differential equation i.e. all uniaxial solutions with this symmetry are of the radial-hedgehog type, modulo an orthogonal transformation. Our method of proof 
is purely based on the governing partial differential equations for $s$, $f$ and $g$. 
We also show that the ``escape in third dimension'' director field cannot correspond to a uniaxial solution, since we cannot find a $s$ compatible with this director. This has been previously reported in the literature under the assumption that $s$ is independent of $z$ \cite{lamy2015uniaxial}; our proof again does not use such assumptions and only relies on the equations. Our last result in the elastic isotropic case concerns uniaxial solutions that have $\mathbf{e}_z$, the unit-vector in the $z$-direction, as an eigenvector; we use a basis representation of $\mathbf{Q}$-tensors in terms of five scalar functions, two of which necessarily vanish when $\mathbf{e}_z$ is a fixed eigenvector. We analyse the governing equations for the remaining three scalar functions to exclude the existence of non-trivial solutions of this type. Our result is not subsumed by results in \cite{lamy2015uniaxial} where the author defines reduced problems in terms of invariance in one direction i.e.$\mathbf{v} \cdot \nabla \mathbf{Q}=0$ for some unit-vector $\mathbf{v}$ and our method of proof is different, which doesn't rely on the ``extra equation''.
Our last result focuses on an anisotropic elastic energy density in the LdG energy functional. The anisotropic term in the Euler-Lagrange equations is a non-trivial technical challenge. 
We apply the same techniques as in the elastic isotropic case, to compute the projections of these equations in three different spaces, and manipulate these projections to show that if $s=s(r)$ and if $f$ and $g$ are independent of $r$, then we must necessarily have $f = \varphi$, $g = \theta$ and $s$ is a solution of an explicit second-order nonlinear ordinary differential equation. This is exactly the anisotropic ``radial-hedgehog'' solution which has been reported in \cite{mkaddemgartland1999} but ours is the first rigorous analysis of uniaxial equilibria in the anisotropic LdG case. 

The paper is organized as follows. In Section~\ref{sec:prelim}, we introduce the basic mathematical preliminaries for the Landau-de Gennes theory. In Sections~\ref{sec:uniaxial1}, \ref{sec:extra_equation}, \ref{sec:dimension_reduction}, we focus on the elastic isotropic case and in section~\ref{sec:anisotropic}, we study an anisotropic LdG elastic energy density. In Section~\ref{sec:conclusions}, we present our conclusions and future perspectives.

\section{Preliminaries}
\label{sec:prelim}

We consider the LdG theory in the absence of any external fields and surface energies \cite{amaz, lamy2015uniaxial, henaomajumdar2012}. The LdG energy is a nonlinear, non-convex functional of the LdG $\Qvec$-tensor and its spatial derivatives; the LdG free energy is given by \cite{dg}
\begin{equation}\label{LdG}
\mathcal{F}[\Qvec] = \int_{\Omega} f_{\rm b}(\Qvec) + f_{\rm el}(\Qvec, \nabla \Qvec) \dd \x 
\end{equation}
with $f_{\rm b}$ and $f_{\rm el}$ the bulk and elastic energy densities, given by
\begin{equation}
f_{\rm b} =  \frac{\alpha(T - T^{*})}{2} \tr(\Qvec^2) - \frac{b^2}{3}\tr(\Qvec^3) + \frac{c^2}{4} (\tr(\Qvec^2))^2,
\end{equation}
\begin{equation}
f_{\rm el} = \frac{L}{2} \left ( | \nabla \Qvec|^2 + L_2 \left(\textrm{div} \Qvec \right)^2 \right) , 
\end{equation}
where 
$\alpha, b^2, c^2 > 0$ are material-dependent constants, 
$T$ is the absolute temperature, and $T^{*}$ is the supercooling temperature below which the isotropic phase $\Qvec = 0$ loses its stability.
Further, $L > 0$ is an elastic constant and $L_2$ is the ``elastic anisotropy'' parameter. In the remainder of this section, we set $L_2=0$, labelled as the ``elastic isotropic'' case and we re-visit the ``anisotropic'' $L_2 \neq 0$ case in the last section.

It is convenient to nondimensionalize (\ref{LdG}) in the following way. Define $\xi = \sqrt{\frac{27 c^2 L}{b^4}}$ as a characteristic length and rescale the variables by \cite{henaomajumdarpisante2017}
\begin{equation}
\tilde{\x} = \frac{\x}{\xi}, \quad \tilde{\Qvec} = \sqrt{\frac{27 c^4}{2 b^4}} \Qvec, \quad \tilde{\mathcal{F}} = \sqrt{\frac{27 c^6}{4b^4L^3}} \mathcal{F}.
\end{equation} 
Dropping the superscript for convenience,  the dimensionless LdG functional can be written as
\begin{equation}\label{eq_dimless}
\mathcal{F}[\Qvec] = \int_{\Omega} \frac{t}{2} \tr(\Qvec^2) - \sqrt{6} \tr(\Qvec^3) + \frac{1}{2} (\tr(\Qvec^2))^2
+ \frac{1}{2} |\nabla \Qvec|^2 \dd \mathbf{x},
\end{equation}
where $t = \dfrac{27 \alpha(T - T^{*}) c^2}{b^4}$ is the reduced temperature. 

We work with temperatures below the nematic-isotropic transition temperature, that is $t \leq 1$. It can be verified that $f_{\rm b}$ attains its minimum on the set of $\Qvec$-tensors given by \cite{majumdar2010equilibrium}
\begin{equation}
\Qvec_{min} = \left\{ \Qvec = s_{+} (\n \otimes \n - \frac{1}{3} \mathbf{I}), \quad \n \in \mathbb{S}^2 \right\}
\end{equation}
for $t \leq 1$, where
\begin{equation}
s_{+} = \sqrt{\frac{3}{2}} \cdot \frac{3  + \sqrt{9 - 8t}}{4}. 
\end{equation}

The LdG equilibria or LdG critical points are classical solutions of the associated Euler-Lagrange equations \cite{amaz}
\begin{equation}\label{E-L}
\Delta \Qvec_{ij} = t \Qvec_{ij} - 3 \sqrt{6} \left( \Qvec_{ik}\Qvec_{kj} - \frac{1}{3} \delta_{ij} \tr(\Qvec^2) \right) + 2 \Qvec_{ij} \tr(\Qvec^2),
\end{equation}
where the term $\sqrt{6} \delta_{ij} \tr(\Qvec^2)$ is a Lagrange multiplier accounting for the tracelessness constraint $\tr(\Qvec) = 0$. This is a system of five elliptic, nonlinear, coupled partial differential equations.
The question of interest is - do we have purely uniaxial solutions of the form (\ref{eq:uniaxial}) of the system (\ref{E-L})?

\section{Elastic Isotropic Case}

\subsection{Uniaxial Solutions with Specified Symmetries}
\label{sec:uniaxial1}

We recall the governing partial differential equations for uniaxial solutions of the one-constant LdG Euler-Lagrange equations from \cite{lamy2015uniaxial}. 
We are seeking nontrivial  uniaxial solutions 
\begin{equation}\label{Q_U}
\Qvec(x) = s(\x) (\n(\x) \otimes \n(\x) - \frac{1}{3} \mathbf{I}), \quad \x \in \Omega
\end{equation}
for the Euler-Lagrange equations (\ref{E-L}) in $\mathbb{R}^3$. 

Substituting (\ref{Q_U}) into (\ref{E-L}), we get
\begin{equation}
\tr(\Qvec^2) = \frac{2}{3} s^2, \quad  \Qvec_{ik}\Qvec_{kj} - \frac{1}{3} \delta_{ij} \tr(\Qvec^2) = \frac{1}{3} s^2 (\n \otimes \n - \frac{1}{3} {\mathbf{I}}),
\end{equation}
\begin{equation}
\Delta \Qvec = \Delta s (\n \otimes \n - \frac{1}{3} \mathbf{I}) + 4 \n \odot (\nabla s \cdot \nabla) \n) + 2 s (\n \odot (\Delta \n)) + 2 s (\partial_k \n \otimes \partial_k \n),
\end{equation}
where $\odot$ denotes the symmetric tensor product, i.e. $(\n \odot \m)_{ij} = (n_i m_j + n_j m_i) / 2$.

Following \cite{lamy2015uniaxial} and rearranging the terms, we get
\begin{equation}
M_1 + M_2 + M_3 = 0,
\end{equation}
where
\begin{equation}\label{decomp1}
\begin{aligned}
& M_1 = \Bigl( \Delta s - 3 |\nabla \n|^2 s - (ts - \sqrt{6} s^2 + \frac{4}{3} s^3) \Bigr) \left(\n \otimes \n  - \frac{1}{3} \mathbf{I} \right), \\
& M_2 = 2 \n \odot \Bigl( s \Delta \n + 2 (\nabla s \cdot \nabla) \n + s |\nabla \n|^2 \n \Bigr), \\ 
& M_3 = s \Bigl( 2 \sum_{k = 1}^{3}\partial_k \n \otimes \partial_k \n + |\nabla \n|^2 \left( \n \otimes \n - \mathbf{I} \right) \Bigr). \\
\end{aligned}
\end{equation}

 The unit-length constraint $|\n|^2 = 1$ implies that
\begin{equation}
\begin{aligned}
& (\nabla \n)^{\rm T} \n = \bf{0}, \\
& \n \cdot \Delta \n + |\nabla \n|^2 = 0 \\
\end{aligned}
\end{equation}
for $(\nabla \n)_{ij} = \pp_j n_i = n_{i, j}$, so that
\begin{equation}
\begin{aligned}
 \n \cdot (s \Delta \n + 2 (\nabla s \cdot \nabla) \n + s |\nabla \n|^2 \n) &  = s (\n \cdot \Delta \n + (\n \cdot \n) |\nabla \n|^2  ) + 2 \n \cdot \left((\nabla \n) \nabla s \right) \\ 
                                                                       & =  s (\n \cdot \Delta \n + |\nabla \n|^2) +  2 \left((\nabla \n)^{\mathrm{T}} \n \right) \cdot \nabla s = 0. \\
\end{aligned}
\end{equation}

Thus we have
\begin{equation}
\begin{aligned}
  & M_1 \in  V_1 = \text{span} \left \{ \n \odot \n - \frac{1}{3}\mathbf{I} \right \}, \\
  & M_2 \in  V_2 = \text{span} \left \{ \n \odot \vvec ~ | ~ \vvec  \in \n^{\perp} \right\}, \\
  & M_3 \in  V_3 = \text{span} \left \{ \vvec  \odot \w ~ | ~ \vvec , \w \in \n^{\perp}, \tr ( \vvec  \odot \w ) = 0 \right\}. \\
\end{aligned}
\end{equation}

Since $M_1, M_2, M_3$ are $3 \times 3$ symmetric traceless  pairwise orthogonal tensors for the usual scalar product on $M_3(\mathbb{R})$, we deduce
\begin{equation}
M_1 = M_2 = M_3 = 0.
\end{equation}

Therefore, $s$ and $\n$ are solutions of \cite{lamy2015uniaxial}
\begin{equation}\label{eq_sn}
\begin{cases}
& \Delta s = 3 |\nabla\n|^2 s + ts - \sqrt{6} s^2 + \frac{4}{3} s^3 \\
& s \Delta \n + 2 (\nabla s \cdot \nabla) \n + s |\nabla \n|^2 \n = 0, \\ 
\end{cases}
\end{equation}
and in the regions where $s$ does not vanish, $\n$ satisfies the extra equation
\begin{equation}\label{eq_n}
2 \sum_{k = 1}^{3}\partial_k \n \otimes \partial_k \n + |\nabla \n|^2 (\n \otimes \n - \mathbf{I}) = \mathrm{0}.
\end{equation}

In what follows, we often work with spherical polar coordinates defined by
\begin{equation}
\label{coord}
\x = \left( r \cos\theta \sin \varphi, r \sin\theta \sin\varphi, r \cos \varphi \right),
\end{equation}
where $0\leq r < \infty$, $0 \leq \varphi \leq \pi$ and $0\leq \theta < 2 \pi$. Our first result, Proposition (\ref{prop:1}), concerns uniaxial solutions with special symmetries as described below.

Assume that (\ref{Q_U}) is a uniaxial solution with
\begin{equation}
 \n(\x) = \left( \sin f(\x) \cos g(\x),~ \sin f(\x) \sin g(\x),~ \cos f(\x) \right).
\end{equation}

Define
\begin{equation}
\begin{aligned}
&  \m(\x) = \left( \cos f( \x) \cos g( \x),~ \cos f(\x) \sin g( \x),~ - \sin f( \x) \right), \\
&  \p(\x) = \left( - \sin g(\x),~ \cos g( \x),~ 0 \right), \\
\end{aligned}
\end{equation}
then $\n$, $\m$, $\p$ are pairwise orthogonal and
\begin{equation}
\n \otimes \n + \m \otimes \m + \p \otimes \p = \mathbf{I}.
\end{equation}

Direct calculations show that
\begin{equation}
\begin{aligned}
& \pp_{r} \n = \frac{\partial f}{\partial r} \m +  \frac{\partial g}{\partial r} \sin f ~ \p,  \\
& \pp_{\varphi} \n = \frac{\partial f}{\partial \varphi} \m +  \frac{\partial g}{\partial \varphi} \sin f ~ \p, \\
& \pp_{\theta} \n = \frac{\partial f}{\partial \theta} \m +  \frac{\partial g}{\partial \theta} \sin f ~ \p,  \\
\end{aligned}
\end{equation}
and
\begin{equation}
\begin{aligned}
 |\nabla \n|^2 & = |\pp_r \n|^2 + \frac{1}{r^2} |\pp_{\varphi} \n|^2 + \frac{1}{r^2 \sin^2 \varphi} |\pp_{\theta} \n|^2 =  |\nabla f|^2 + |\nabla g|^2 \sin^2 f. \\
\end{aligned} 
\end{equation}

Since
\begin{equation}
\begin{aligned}
& \sum_{k = 1}^{3}\partial_k \n \otimes \partial_k \n = \pp_r \n \otimes \pp_r \n + \frac{1}{r^2} \pp_{\varphi} \n \otimes \pp_{\varphi} \n + \frac{1}{r^2 \sin^2 \varphi} \pp_{\theta} \n \otimes \pp_{\theta} \n \\
                                                 & \qquad \qquad \qquad = |\nabla f|^2 \m \otimes \m + |\nabla g|^2 \sin^2 f ~ \p \otimes \p  + 2 \nabla f \cdot \nabla g \sin f \m \odot \p, \\ 
\end{aligned}
\end{equation} 
we have the following from (\ref{eq_n}), 
\begin{equation}
\begin{aligned}
& 2 \sum_{k = 1}^{3}\partial_k \n \otimes \partial_k \n - |\nabla \n|^2 (\n \otimes \n - \mathbf{I}) \\
& \qquad = \left( |\nabla f|^2 - |\nabla g|^2 \sin^2 f \right) (\m \otimes \m - \p \otimes \p) + 4 \nabla f \cdot \nabla g \sin f ~ \m \odot \p.
\end{aligned}
\end{equation}

Since $\m \otimes \m - \p \otimes \p$ and $\m \odot \p$ are orthogonal for the usual scalar product on $M_3 (\mathbb{R})$, in the region where $s$ does not vanish,  $f$ and $g$ satisfy
\begin{equation}\label{eq_fg}
\begin{cases}
& \nabla f \cdot \nabla g = 0 \\
& |\nabla f|^2 = |\nabla g|^2 \sin^2 f.  \\
\end{cases}
\end{equation}

We manipulate the second equation in  (\ref{eq_sn}) to get 
%
%
\begin{equation}
\left( s (\Delta f - |\nabla g|^2 \sin f \cos f ) + 2 \nabla s \cdot \nabla f  \right) \m + \left( s \Delta g \sin f + 2 ( \nabla s \cdot \nabla g) \sin f \right) \p = 0.
\end{equation}
Since $\m$ and $\p$ are orthogonal, we have
\begin{equation}
\begin{cases}
& s (\Delta f - |\nabla g|^2 \sin f \cos f ) + 2 \nabla s \cdot \nabla f = 0 \\
& s \Delta g + 2 \nabla s \cdot \nabla g  = 0. \\
\end{cases}
\end{equation}


Thus the partial differential equations for $s$, $f$, $g$ are:
\begin{equation}\label{s_f_g}
\begin{cases}
&  \Delta s = 3 \left( |\nabla f|^2 + |\nabla g|^2 \sin^2 f \right) s + \psi(s)  \\
& s \left( \Delta f - |\nabla g|^2 \sin f \cos f \right) + 2 \nabla s \cdot \nabla f = 0 \\
& s \Delta g + 2 \nabla s \cdot \nabla g  = 0 \\
& s \left( \nabla f \cdot \nabla g \right) = 0 \\
& s \left( |\nabla f|^2 - |\nabla g|^2 \sin^2 f \right) = 0,  \\
\end{cases}
\end{equation}
where
\begin{equation}
\psi(s) = ts - \sqrt{6} s^2 + \frac{4}{3} s^3.
\end{equation}

\begin{proposition}
\label{prop:1} If 
\begin{equation}\label{Ansatz_1}
\Qvec(r, \theta, \varphi) = s(r, \theta, \varphi) \left( \n(\theta, \varphi) \otimes \n(\theta, \varphi) - \frac{1}{3} \mathbf{I} \right)
\end{equation}
is a non-trivial uniaxial solution of (\ref{E-L}) with
\begin{equation}\label{Ansatz_n}
 \n (\theta, \varphi) = \left( \sin f(\varphi) \cos g(\theta),~ \sin f(\varphi) \sin g(\theta),~ \cos f(\varphi) \right),
\end{equation}
then 
\begin{equation}
f(\varphi) = \pm \varphi, \quad \frac{\dd g}{\dd \theta}= \pm 1
\end{equation}
and $s$ satisfies
\begin{equation}\label{eq_h}
 s''(r) + \frac{2}{r} s'(r) = \frac{6}{r^2} s(r)  + ts - \sqrt{6} s^2 + \frac{4}{3} s^3. 
\end{equation}
\end{proposition}



\begin{rem}
Since $\nabla f \cdot \nabla g = 0$, we only need to assume that $g = g(\theta)$ or $f = f(\varphi)$ in (\ref{Ansatz_n}).
\end{rem}

\begin{proof}

From $|\nabla f|^2 = |\nabla g|^2 \sin^2 f$, we have
\begin{equation}\label{eq:2}
\left(\frac{\dd f}{\dd \varphi}\right)^2 = \frac{\sin^2 f(\varphi)}{\sin^2 \varphi}\left(\frac{\dd g}{\dd \theta}\right)^2.
\end{equation}

Since we have assumed that $f = f(\varphi)$ and $g = g(\theta)$, equation (\ref{eq:2}) further simplifies to
\begin{equation}\label{f_g_s_v}
\begin{aligned}
& \frac{\sin \varphi}{\sin f(\varphi)} \frac{\dd f}{\dd \varphi} = C_1, \quad \frac{\dd g}{\dd \theta}  = \pm C_1, \\
\end{aligned}
\end{equation}
where $C_1$ is some constant.

From (\ref{f_g_s_v}), we have 
\begin{equation}
\begin{aligned}
\dfrac{\dd^2 f}{\dd \varphi^2} = C_1^2 \dfrac{\cos f \sin f}{\sin^2 \varphi} - C_1 \dfrac{\cos \varphi \sin f}{\sin^2 \varphi}.
\end{aligned}
\end{equation}
Hence,
\begin{equation}
\nabla s \cdot \nabla f =  - \frac{1}{2} s (\Delta f - |\nabla g|^2 \sin f \cos f) = -\dfrac{s}{2r^2} \left(\dfrac{\dd^2 f}{\dd \varphi^2} + \dfrac{\cos \varphi}{\sin \varphi} \dfrac{\dd f}{\dd \varphi} - C_1^2 \dfrac{\sin f \cos f}{\sin^2 \varphi} \right) = 0,  
\end{equation}
which implies that $\pp_{\varphi} s = 0$.

Similarly, from (\ref{f_g_s_v}), we have
\begin{equation}
\nabla s \cdot \nabla g = -\frac{1}{2} s \Delta g =  - \dfrac{s}{2 r^2 \sin^2 \varphi} \dfrac{\dd^2 g}{\dd \theta^2}   = 0,
\end{equation}
which implies $\pp_{\theta} s = 0$.

As we have shown that $s = s(r)$, the first equation in (\ref{eq_sn}) requires that $|\nabla \n|^2$ is independent of $\theta$ and $\varphi$, i.e. 
\begin{equation}
|\nabla \n|^2  = \frac{2}{r^2} \left(\frac{\dd f}{\dd \varphi}\right)^2 = C(r),
\end{equation}
where $C(r)$ is independent with $\theta$ and $\varphi$. Hence, 
\begin{equation}
\frac{\dd f}{\dd \varphi} = C_2
\end{equation}
for some constant $C_2$.

Recalling (\ref{f_g_s_v}), we have
\begin{equation}\label{eq_C1C2}
C_1 \sin (C_2 \varphi + C_3) = C_2 \sin \varphi,
\end{equation}
where $C_3$ is a real constant.
Computing the second derivatives of both sides, we have
\begin{equation}
- C_1 C_2^2 \sin(C_2 \varphi + C_3) = - C_2 \sin \varphi = -C_1 \sin (C_2 \varphi + C_3),
\end{equation}
which implies that $C_2^2 = 1$. 

Referring back to (\ref{eq_C1C2}), we have
\begin{equation}
C_1 \sin (\varphi + C_3) = \sin \varphi \quad \text{or} \quad C_1 \sin (- \varphi + C_3) = - \sin \varphi,
\end{equation}
which implies that $C_1 = 1, C_3 = k \pi$ ($k$ is even) or  $C_1 = - 1, C_3 = k \pi$ ($k$ is odd).

Since $\n$ is equivalent to $- \n$ in the LdG theory, we can take $C_3 = 0$ without loss of generality. Hence, we get 
\begin{equation}\label{C1C2}
f(\varphi) = \pm \varphi, \quad \dfrac{\dd g}{\dd \theta} = \pm 1.
\end{equation}


The existence of a solution for equation (\ref{eq_h}) with suitable boundary conditions has been proven in several papers e.g. \cite{majumdar2012radial, lamy2013some, ignat2014uniqueness}. 
\end{proof}

\begin{corollary}
\label{c1}
Let $\Qvec$ be a smooth non-trivial uniaxial solution of (\ref{E-L}) of the form (\ref{Q_U}) with
$$\n = \left( \sin f \cos g, \sin f \sin g, \cos f \right). $$

If $f= f(\varphi)$ and $s = s(r)$ with $s \neq 0$ for $r > 0$, then we necessarily have that
\begin{equation}
\label{eq:c1}
g(\theta) = \pm \theta + C, \quad f(\varphi) = \pm \varphi
\end{equation}
for some real constant $C$.
\end{corollary}

\begin{proof}
If $s \neq 0$, then we have $\pp_{\varphi} g = 0$ from $\nabla f \cdot \nabla g = 0$ and

\begin{equation}\label{g_fs_s}
\begin{cases}
& (\pp_r g)^2 + \dfrac{1}{r^2 \sin^2 \varphi} (\pp_{\theta} g)^2 = \dfrac{1}{r^2 \sin^2 \varphi}\left( \dfrac{\dd f}{\dd \varphi} \right)^2 \\ 
& s \Bigl( \pp_r^2 g + \dfrac{2}{r} \pp_r g + \dfrac{1}{r^2 \sin^2 \varphi} \pp_{\theta}^2 g \Bigr) + 2 \pp_r s  \pp_r g = 0. \\
\end{cases}
\end{equation}

For fixed $r_0 > 0$, since the uniaxial $\Qvec$ is smooth and $\Qvec \neq 0$ on $B(r_0, \delta)$ for some $\delta > 0$, we can have that  $s$ and $\n$ are smooth on $B(r_0, \delta)$ \cite{amaz, lamy2015uniaxial}. Hence, on $B(r_0, \delta)$, we have:
\begin{equation}\label{g_exp}
g(r, \theta) = g_0(\theta) + g_1 (\theta) (r - r_0) + g_2(\theta) (r - r_0)^2 + O((r - r_0)^3), \quad |r - r_0| < \delta.
\end{equation}
Substituting (\ref{g_exp}) into the first equation in (\ref{g_fs_s}), and letting $r \rightarrow r_0$, we have
\begin{equation}
  g_1^2 = \frac{1}{r_0^2 \sin^2 \varphi} \left(\left( \frac{\dd f}{\dd \varphi} \right)^2 - \left( \frac{\dd g_0}{\dd \theta} \right)^2 \right). \\
\end{equation}

Since $g$ is independent with $\varphi$, we have $g_1 = 0$. 
By the arbitrariness of $r_0$, we get $\pp_r g = 0$ for $\forall r > 0$. Hence
\begin{equation}
\frac{\dd g}{\dd \theta} = \pm C_1, \quad \frac{\dd f}{\dd \varphi} =  C_1 
\end{equation} 
for some real constant $C_1$.

Recalling the second equation in (\ref{s_f_g}), we have
\begin{equation}
 C_1 \sin (C_1 \varphi + C_3) \cos(C_1 \varphi + C_3) = \sin \varphi \cos \varphi, \quad \forall \varphi
\end{equation}
for some constant C$_3$. Hence, we have $C_1 = \pm 1$ by taking $C_3 = 0$ without loss of generality.

\end{proof}

\begin{rem}
A solution $(s, f, g)$ of the system of equations (\ref{s_f_g}) can be regarded as a critical point of the functional 
\begin{equation}\label{fun_sfg}
E(x, \mathbf{u}(x), {\rm D} \mathbf{u}(x)) = \int_{\Omega} \left(  \frac{1}{3} ts^2 - \frac{2 \sqrt{6}}{9} s^3 + \frac{2}{9} s^4 + \frac{1}{3} |\nabla s|^2  + s^2 (|\nabla f|^2 + |\nabla g|^2 \sin^2 f) \right) \dd \x,
\end{equation}
in the constrained admissible class 
\begin{equation}\label{ad_sfg}
\mathcal{A}_u := \left\{ s, f, g \in W^{1,2}(\Omega, \mathbb{R}) ~|~ s (\nabla f \cdot \nabla g) = 0,~ s (|\nabla f|^2 - |\nabla g|^2 \sin^2 f) = 0   \right\},
\end{equation} 
subject to Dirichlet boundary conditions, where $\mathbf{u} = (s, f, g)$. The constraints in (\ref{ad_sfg}) 
are nonholonomic  \cite{giaquinta2013calculus} and are difficult to deal with.
\end{rem}

Indeed, according to the calculations in \cite{Biscari2006}, it is difficult to find unit-vector fields $\n$ that solve (\ref{eq_n}).  
In the remainder of this subsection, we discuss the ``third dimension escape'' solution \cite{Cladis1972} in greater detail. 
The ``third dimension escape'' solution is known to be a non-trivial explicit solution of the extra equation (\ref{eq_n}) \cite{Biscari2006, lamy2015uniaxial}. However, we cannot have an order parameter $s$ such that $(s, \n)$ solves (\ref{eq_sn}).  Theorem 4.1 in \cite{lamy2015uniaxial} suggests that this solution cannot be purely uniaxial if $\pp_z s = 0$. Here, we provide an alternative proof by using (\ref{s_f_g}), without assuming $\pp_z s = 0$.

The ``escape into third dimension'' uniaxial director is given in cylindrical coordinates $(\rho, \theta, z)$ by
\begin{equation}
\label{third_dim}
\n(\rho, \theta, z) = \cos \Psi(\rho) \mathbf{e}_r + \sin \Psi(\rho) \mathbf{e}_z \quad \text{with} \quad \rho \frac{\dd \Psi}{\dd \rho} = \cos \Psi,
\end{equation}
where $\e_r = (\cos \theta, \sin \theta, 0)$, $\e_z = (0, 0, 1) \in \mathbb{R}^3, 0 \leq \rho \leq 1$. Hence,
\begin{equation}
\n(r, \theta, \varphi) = (\sin f \cos g, \sin f \sin g, \cos f)
\end{equation}
with
\begin{equation}\label{es_fg}
f = \frac{\pi}{2} - \Psi(r \sin \varphi), \quad g = \theta.
\end{equation}

Therefore,
\begin{equation}
\begin{aligned}
& \pp_r f = - \frac{1}{r} \cos \Psi, \quad \pp_{\varphi} f = -\frac{\cos \varphi}{\sin \varphi} \cos \Psi, \\
& \pp_r^2 f = \frac{1}{r^2} (\cos \Psi + \sin \Psi \cos \Psi), \quad \pp_{\varphi}^2 f = \frac{1}{\sin^2 \varphi} \cos \Psi + \frac{\cos^2 \varphi}{\sin^2 \varphi} \sin \Psi \cos \Psi. \\
\end{aligned}
\end{equation}
Direct calculations show that (\ref{es_fg}) satisfies (\ref{eq_n}) and 
\begin{equation}
\Delta f - |\nabla f|^2 \cos f / \sin f = 0, \quad \Delta g = 0.
\end{equation}

Assume there exists a scalar order parameter $s$ such that the pair $(s, \n)$ satisfies (\ref{eq_sn}), then (\ref{eq_sn}) requires that $s$ satisfies 
\begin{equation}\label{exp_no_s}
\begin{aligned}
& \Delta s = \frac{6}{r^2} \dfrac{\cos^2 \Psi}{\sin^2 \varphi} s + \psi(s), \\ 
& \nabla s \cdot \nabla f = 0 ~~ \Rightarrow ~~ \pp_r s  = - \frac{1}{r} \dfrac{\cos \varphi}{\sin \varphi} \pp_{\varphi} s, \\
& \nabla s \cdot \nabla g = 0 ~~ \Rightarrow ~~ \pp_{\theta} s = 0. \\
\end{aligned}
\end{equation}
In Cartesian coordinates, we have
\begin{equation}
\begin{aligned}
& \pp_x s = \cos \theta \sin \varphi \pp_r s + \frac{\cos \theta \cos \varphi}{r} \pp_{\varphi} s = 0, \\
& \pp_y s = \sin \theta \sin \varphi \pp_r s + \frac{\sin \theta \cos \varphi}{r} \pp_{\varphi} s = 0, \\
& \pp_z s = \cos \varphi \pp_r s - \frac{\sin \varphi}{r} \pp_{\varphi} s = - \frac{1}{r \sin \varphi} \pp_{\varphi} s. \\
\end{aligned}
\end{equation}
Hence, the first equation of (\ref{exp_no_s}) can be recast as
\begin{equation}
\pp_z^2 s = \frac{6}{r^2} \frac{\cos^2 \Psi(x, y)}{\sin^2 \varphi} s + \psi(s). 
\end{equation}

By taking derivatives with respect to $x$ and $y$ on both sides, we have 
\begin{equation}
\begin{aligned}
& s \frac{\pp }{\pp x} \left(\frac{6}{r^2} \frac{\cos^2 \Psi(x, y)}{\sin^2 \varphi} \right) = 0  \\
& s \frac{\pp}{\pp y} \left( \frac{6}{r^2} \frac{\cos^2 \Psi(x, y)}{\sin^2 \varphi} \right) = 0 \\
\end{aligned}
\quad \quad \quad \forall x,~y, 
\end{equation}
which implies that $s \equiv 0$. Hence, we cannot find a non-trivial $s$ for which $(s, \n)$, with $\n$ as given in (\ref{third_dim}), is a solution of (\ref{eq_sn}).



\subsection{A New Perspective for the Extra Equation (\ref{eq_n})}
\label{sec:extra_equation}
Consider
\begin{equation}\label{uni_sn}
\Qvec = s(\x) \Bigl( \n(\x) \otimes \n(\x) - \frac{1}{3} \mathbf{I} \Bigr) + \beta(\x) \Bigl( \m (\x) \otimes \m (\x) - \p (\x) \otimes \p(\x) \Bigr), \quad \forall \x \in \Omega \subset  \mathbb{R}^3, 
\end{equation}
where $\n$ is the leading eigenvector of $\Qvec$ (with the largest eigenvalue in terms of magnitude), $0 \leq |\beta| \leq \frac{1}{3} |s|$. In the case that the eigenvalues of $\Qvec$ are $\frac{2|s|}{3}, 0, -\frac{2|s|}{3}$ respectively, we define the eigenvector
corresponding to the eigenvalue $\frac{2|s|}{3}$ as the leading eigenvector, which implies that for $s < 0$, we have $|\beta| < \frac{1}{3} |s|$.
Inspired by \cite{Biscari2006}, we have the following result:
\begin{proposition}
Let $\Qvec$ 
be a global minimizer of LdG free energy in the admissible class $\mathcal{A}$, for which the leading eigenvector $\n(\x)$ satisfies the extra equation (\ref{eq_n}) in $\Omega$,
subject to uniaxial boundary conditions 
\begin{equation}
s(\x) = s_{+} > 0, \quad \beta(\x) = 0, \quad \forall \x \in \pp \Omega.
\end{equation}
Then $\Qvec$ is necessarily uniaxial with $\beta \equiv 0$ everywhere in $\Omega$ for $t \geq 0$.
\end{proposition}
\begin{rem}
For a director field $\n(\x)$ in $\Omega$, we may not have critical points of LdG free energy in the admissible class $\mathcal{A}$.
\end{rem}

\begin{proof}

For $\Qvec$ of the form (\ref{uni_sn}), we can check that
\begin{equation}
\begin{aligned}
|\nabla \Qvec|^2 & = \frac{2}{3}|\nabla s|^2 + 2 |\nabla \beta|^2 + 2 s^2 |\nabla \n|^2 \\
            & + 2 \beta^2 (|\nabla \n|^2 + 4 |(\nabla \m)^{\rm T} \p|^2) - 4 s \beta(\p \cdot G \p - \p \cdot G \m), \\
\end{aligned}
\end{equation}
where $G = (\nabla \n)(\nabla \n)^{\rm T} = \sum_{k = 1}^3 \pp_k \n \otimes \pp_k \n$.

The extra equation (\ref{eq_n}) can be written as
\begin{equation}
G = \frac{1}{2} |\nabla \n|^2 \m \otimes \m + \frac{1}{2} |\nabla \n|^2 \p \otimes \p.
\end{equation}
Since $\m$ and $\p$ are orthogonal, we easily obtain
\begin{equation}
G \m = \frac{1}{2} |\nabla \n|^2 \m, \quad G \p = \frac{1}{2} |\nabla \n|^2 \p.
\end{equation}
Hence, if the leading eigenvector $\n$ satisfies (\ref{eq_n}), then
\begin{equation}
|\nabla \Qvec|^2  = \frac{2}{3}|\nabla s|^2 + 2 |\nabla \beta|^2 + 2 s^2 |\nabla \n|^2 + 2 \beta^2 \left( |\nabla \n|^2 + 4 |(\nabla \m)^{\rm T} \p|^2 \right).
\end{equation}

Substituting  (\ref{uni_sn}) into (\ref{eq_dimless}) and using the above reduction for the one-constant elastic energy density, the LdG energy in this restricted class is
\begin{equation}\label{fun_snb}
\begin{aligned}
  \mathcal{F}(s, \beta, \n , \m) & = \int \left( \frac{t}{2} \left( \frac{6}{9} s^2 + 2 \beta^2 \right) + \sqrt{6} \left( 2 \beta^2 - \frac{2}{9} s^2 \right) s + \frac{2}{9} \left( s^2 + 3 \beta^2 \right)^2 \right) \\
                               & + \frac{1}{3}|\nabla s|^2 + |\nabla \beta|^2 + s^2 |\nabla \n|^2 + \beta^2 \left( |\nabla \n|^2 + 4 |(\nabla \m)^{\rm T} \p|^2 \right) \dd \x. \\
\end{aligned}
\end{equation}

The associated Euler-Lagrange equations for $s$ and $\beta$ are
\begin{equation}
\begin{cases}
& \Delta s = 3 |\nabla \n|^2 s + ts - \sqrt{6} s^2 + \frac{4}{3}s^3 + 4 \beta^2 s + 3 \sqrt{6} \beta^2 \\
& \Delta \beta = \left( |\nabla \n|^2 + 4 |(\nabla \m)^{\rm T} \p|^2 + t + 2\sqrt{6} s + \frac{4}{3}s^2 \right) \beta +  4 \beta^3. \\
\end{cases}
\end{equation}

We note that
\begin{equation}\label{eq_beta2}
\begin{aligned}
\Delta \beta^2 & = 2 (\nabla \beta \cdot \nabla \beta + \beta \Delta \beta) \\
              & = 2 \left( |\nabla \beta|^2 + \bigl( |\nabla \n|^2 + 4 |(\nabla \m)^{\rm T} \p|^2 + t + 2\sqrt{6} s + \frac{4}{3}s^2 \bigr) \beta^2 +  4 \beta^4 \right). \\ 
\end{aligned}
\end{equation}

In order to get the desired result, we firstly show that $s \geq 0$, which can be proved by contradiction. The proof is similar to the proof of Lemma 2 in Ref. \cite{majumdar2010equilibrium}. Let $\Omega^{*} = \{\x \in \Omega;~~s(\x) < 0 \}$ be a measurable interior subset of $\Omega$. The boundary condition implies that the subset $\Omega^{*}$ does not intersect $\pp \Omega$. Then we can consider the perturbation
\begin{equation}
\widetilde{\Qvec} = 
\begin{cases}
& s(\x) \Bigl( \n(\x) \otimes \n(\x) - \dfrac{1}{3} \mathbf{I} \Bigr) + \beta(\x) \Bigl( \m(\x)  \otimes \m(\x)  - \p(\x) \otimes \p(\x) \Bigr), \quad \forall \x \in \Omega \backslash \Omega^{*} \\ 
& - s(\x) \Bigl( \n(\x) \otimes \n(\x) - \dfrac{1}{3} \mathbf{I} \Bigr) + \beta(\x) \Bigl( \m(\x)  \otimes \m(\x)  - \p(\x) \otimes \p(\x) \Bigr), \quad \forall \x \in \Omega^{*}. \\
\end{cases}
\end{equation}
Then $\widetilde{\Qvec} \in \mathcal{A}$ and $\widetilde{\Qvec}$ coincides with $\Qvec$ everywhere outside $\Omega^{*}$, The free energy difference $\mathcal{F}(\widetilde{\Qvec}) - \mathcal{F}(\Qvec)$ is 
\begin{equation}
\begin{aligned}
\mathcal{F}(\widetilde{\Qvec}) - \mathcal{F}(\Qvec) &= \int_{\Omega^{*}} 4\sqrt{6}(\frac{1}{9} s^3 - s \beta^2) \dd \x < 0,
\end{aligned}
\end{equation}
where the last inequality holds because $s < 0$ and $\beta^2 < \dfrac{1}{9} s^2$ for $s < 0$. This contradicts the fact that $\Qvec$ is a global minimizer in the admissible class $\mathcal{A}$. Hence, $\Omega^{*}$ is empty and $s \geq 0$ everywhere in $\Omega$. So for $t \geq 0$,
\begin{equation}
t + 2\sqrt{6} s(\x) + \frac{4}{3}s(\x)^2 \geq 0, \quad \forall \x \in \Omega, 
\end{equation}
 which implies that $\Delta \beta^2 \geq 0$ and $\beta^2$ is subharmonic. 
By the weak maximum principle \cite{han2011basic}, we have $||\beta^2||_{L^{\infty} (\Omega)} \leq ||\beta^2||_{L^{\infty} (\pp \Omega)} = 0$. Hence, $\beta$ is identically zero in $\Omega$ and $\Qvec$ is necessarily uniaxial. 

\end{proof}

\subsection{An alternative approach}
 \label{sec:dimension_reduction}


Let
\begin{equation}
  \begin{aligned}
    \mathcal{S} = \{ \Qvec \in M^{3\times3} (\mathbb{R}) ~|~ \Qvec = \Qvec^{\mathrm{T}}, \tr(\Qvec) = 0 \}. \\
   \end{aligned}
\end{equation}

Consider the following 
basis for $\mathcal{S}$: 
\begin{equation}
\begin{aligned}
& \E_1 = \sqrt{\frac{3}{2}}(\mathbf{e}_z \otimes \mathbf{e}_z - \frac{1}{3} \mathbf{I}), \quad \E_2 = \sqrt{\frac{1}{2}}(\mathbf{e}_x \otimes \mathbf{e}_x - \mathbf{e}_y \otimes \mathbf{e}_y), \quad \E_3 = \sqrt{\frac{1}{2}}(\mathbf{e}_x \otimes \mathbf{e}_y + \mathbf{e}_y \otimes \mathbf{e}_x), \\
& \E_4 = \sqrt{\frac{1}{2}}(\mathbf{e}_x \otimes \mathbf{e}_z + \mathbf{e}_z \otimes \mathbf{e}_x), \quad \E_5 = \sqrt{\frac{1}{2}}(\mathbf{e}_y \otimes \mathbf{e}_z + \mathbf{e}_z \otimes \mathbf{e}_y), \\
\end{aligned}
\end{equation}
where $\e_x = (1, 0, 0),~ \e_y = (0, 1, 0),~ \e_z = (0, 0, 1) \in \mathbb{R}^3$.

For $\forall \Qvec \in \mathcal{S}$:
\begin{equation}
\Qvec(\x) = \sum_{i = 1}^{5} q_i(\x) \E_i , \quad \forall \x \in \mathbb{R}^3,
\end{equation}
thus,
\begin{equation}
\tr(\Qvec^2) = \sum_{i = 1}^5 q_i^2, \quad |\nabla \Qvec|^2 = \sum_{i = 1}^5 |\nabla q_i|^2,
\end{equation}
\begin{equation}
\begin{aligned}
\tr (\Qvec^3) & = \frac{\sqrt{6}}{6}q_1^3 + \frac{3\sqrt{2}}{4}(q_2q_4^2 - q_2q_5^2) - \frac{\sqrt{6}}{2}(q_1q_2^2 + q_1q_3^2) + \frac{\sqrt{6}}{4}(q_1q_4^2 + q_1q_5^2) + \frac{3\sqrt{2}}{2}q_3q_4q_5 \\
& = \frac{\sqrt{6}}{6}q_1^3 - \frac{\sqrt{6}}{2}(q_2^2 + q_3^2)q_1 + (\frac{\sqrt{6}}{4}q_1 + \frac{3\sqrt{2}}{4}q_2) q_4^2 + (\frac{\sqrt{6}}{4} q_1 - \frac{3\sqrt{2}}{4} q_2) q_5^2 + \frac{3\sqrt{2}}{2}q_3q_4q_5. \\
\end{aligned}
\end{equation}

Hence,
the Euler-Lagrange equations for $q_i$ ($i = 1, 2, \ldots, 5$) are given by

\begin{equation}\label{eq_qi_b}
\begin{cases}
  & \Delta q_1 =  \left( t - 6 q_1 + 2 (\sum_{k = 1}^5 q_k^2) \right) q_1 + 3 ( \sum_{k = 1}^5 q_k^2) - \dfrac{9}{2} (q_4^2 + q_5^2) \\
  & \Delta q_2 =  \left( t + 6 q_1 + 2 (\sum_{k = 1}^5 q_k^2) \right) q_2 - \dfrac{3\sqrt{3}}{2}(q_4^2 - q_5^2)  \\
  & \Delta q_3 =  \left( t + 6 q_1 + 2 (\sum_{k = 1}^5 q_k^2) \right) q_3  - 3 \sqrt{3} q_4 q_5 \\
  & \Delta q_4 =  \left( t - 3 q_1 - 3 \sqrt{3} q_2 + 2 (\sum_{k = 1}^5 q_k^2) \right) q_4 - 3 \sqrt{3} q_3 q_5 \\
  & \Delta q_5 =  \left( t - 3 q_1 + 3 \sqrt{3} q_2 + 2 (\sum_{k = 1}^5 q_k^2) \right) q_5 - 3 \sqrt{3} q_3 q_4. \\ 
\end{cases}
\end{equation}

It is known that $\Qvec$ is uniaxial, if and only if
\begin{equation}\label{tr_uni}
\tilde{\beta}(\Qvec) = \left( \tr (\Qvec^2) \right)^3 - 6 \left( \tr \Qvec^3) \right)^2 = 0,
\end{equation}
which can be viewed as the uniaxial constraints of (\ref{eq_qi_b}) (see for example \cite{amaz}).

\begin{proposition}
\label{prop:4.1} 
Let $\Omega \subset \mathbb{R}^3$ be an open set, if
\begin{equation}\label{Ansatz}
\Qvec(\x) = \sum_{i = 1}^{3} q_i(\x) \E_i, \quad \forall \x \in \mathbb{R}^3
\end{equation}
is a uniaxial solution of (\ref{E-L}), then
$\Qvec$ has 
a constant eigenframe in every connected component of $\{ \Qvec \neq 0 \}$. Moreover, if $\Omega$ is connected, then $\Qvec$ has a constant eigenframe in the whole domain. 
\end{proposition}

\begin{rem}
We are considering $\Qvec$-tensors with $q_4 =q_5 = 0$, and show that there are no non-trivial uniaxial solutions of this form with $q_4 = q_5 = 0$.
\end{rem}

\begin{proof}
Let $\Omega_1$ be a connected component of $\{ \Qvec \neq 0 \}$. Since $\Qvec$ is uniaxial and $q_4 =q_5 = 0$, then
\begin{equation}
\tilde{\beta}(\Qvec) = (q_2^2 + q_3^2)(- 3q_1^2 + q_2^2 + q_3^2)^2  = 0,
\end{equation}
which implies that
\begin{equation}
q_2 = q_3 = 0 \quad \text{or} \quad q_1^2 = \frac{1}{3}(q_2^2 + q_3^2).
\end{equation}

If $q_2 = q_3 = 0$ in $\Omega_1$, then $\Qvec = q_1 \E_1$ with a constant eigenframe. 

If $q_1^2 = \frac{1}{3}(q_2^2 + q_3^2)$ in $\Omega_1$, we have
\begin{equation}
\Delta q_i =  ( t + 6 q_1 + 8 q_1^2 ) q_i, \quad i = 1, 2, 3.  
\end{equation}
Since $q_4 = q_5 = 0$, $\Qvec$ can be written as
\begin{equation}\label{Q_2D}
\begin{aligned}
\Qvec & = q_1 \sqrt{\frac{3}{2}}(\mathbf{e}_z \otimes \mathbf{e}_z - \frac{1}{3} \mathbf{I}) + v \sqrt{\frac{1}{2}} (\n \otimes \n - \frac{1}{2} \mathbf{I}_2) \\
  & = q_1 \sqrt{\frac{3}{2}}(\mathbf{e}_z \otimes \mathbf{e}_z - \frac{1}{3} \mathbf{I}) - v \sqrt{\frac{1}{2}} (\p \otimes \p - \frac{1}{2} \mathbf{I}_2), \\
\end{aligned}
\end{equation}
where $\n(\x) \in \mathbb{S}^2$, $\n(\x) \perp \e_z$, $\p(\x) = \e_z \times \n(\x)$, and $\mathbf{I}_2 = \e_x \otimes \e_x + \e_y \otimes \e_y$.

Letting $\n(\x) = a_1 \mathbf{e}_x + a_2 \mathbf{e}_y$, we have
\begin{equation}
q_2 = (a_1^2  - \frac{1}{2}) v, \quad q_3 = a_1 a_2 v.
\end{equation}
Hence,
\begin{equation}
q_1^2 = \frac{1}{3}(q_2^2  + q_3^2) = \frac{1}{3}(a_1^4 + \frac{1}{4} - a_1^2 + a_1^2(1 - a_1^2) ) v^2 = \frac{1}{12} v^2.
\end{equation}

Since $\Qvec \neq 0$ in $\Omega_1$, we have $q_1 \neq 0$ in $\Omega_1$. Hence,
\begin{equation}
v = 2 \sqrt{3} q_1 \quad \text{or} \quad v = -2 \sqrt{3} q_1 \quad  \text{in} ~~ \Omega_1.
\end{equation}
Then from (\ref{Q_2D}), we have
\begin{equation}
\Qvec = \sqrt{\frac{1}{2}}v (\n \otimes \n - \frac{1}{3} \mathbf{I}) \quad \text{or} \quad \Qvec = - \sqrt{\frac{1}{2}}v (\p \otimes \p - \frac{1}{3} \mathbf{I}) \quad \text{in} ~~ \Omega_1,
\end{equation}
which implies that $s = \pm \sqrt{\frac{1}{2}}v = - \sqrt{6} q_1$. Thus $s$ is a solution of 
\begin{equation}
\Delta s = (t - \sqrt{6} s + \frac{4}{3} s^2) s.
\end{equation}
Recalling (\ref{eq_sn}), we have $|\nabla \n|^2 = 0$ or $|\nabla \p|^2 = 0$ in $\Omega_1$. Hence, $\Qvec$ has a constant eigenframe in $\Omega_1$. 

If $\Omega$ is connected, then $\Qvec$ is analytic. Following the proof in Theorem 4.1 (ii) in \cite{lamy2015uniaxial}, we can show that the uniaxial analytic $\Qvec$ has a constant eigenframe in the entire domain $\Omega$.


\end{proof}

\section{Elastic Anisotropic Case}
\label{sec:anisotropic}

Consider the dimensionless LdG free energy with elastic anisotropy
\begin{equation}\label{eq_dimless_AE}
\mathcal{F}[\Qvec] = \int_{\Omega} \frac{t}{2} \tr(\Qvec^2) - \sqrt{6} \tr(\Qvec^3) + \frac{1}{2} (\tr(\Qvec^2))^2
+ \frac{1}{2} |\nabla \Qvec|^2  + \frac{L_2}{2} \Qvec_{ij,j}\Qvec_{ik.k}  \dd \mathbf{x},
\end{equation}
where $L_2 \neq 0$. Then the corresponding Euler-Lagrange equations are
\begin{equation}\label{AE-E-L}
  \begin{aligned}
    &  \Delta \Qvec_{ij} + \dfrac{L_2}{2} \left( \Qvec_{ik,kj} + \Qvec_{jk,ki} - \frac{2}{3} \delta_{ij} \Qvec_{kl, kl} \right) \\
    &    = t \Qvec_{ij} - 3 \sqrt{6} \left( \Qvec_{ik}\Qvec_{kj} - \frac{1}{3} \delta_{ij} \tr(\Qvec^2) \right) + 2 \Qvec_{ij} \tr(\Qvec^2).
  \end{aligned}
\end{equation}

We seek uniaxial solutions of the form (\ref{Q_U}) for the Euler-Lagrange equations (\ref{AE-E-L}). Let
\begin{equation}
  \begin{aligned}
    & V_1 = \mathrm{span} \left\{ \n \odot \n  - \frac{1}{3} \mathbf{I} \right\}, \\
    & V_2 = \mathrm{span} \left\{ \n \odot \vvec ~|~  \vvec \in \n^{\perp} \right\}, \\
    & V_3 = \mathrm{span} \left\{ \vvec \odot \w ~|~ \vvec, \w \in \n^{\perp}, \tr (\vvec \odot \w) = 0 \right\}, \\
    \end{aligned} 
\end{equation}
and  $P_i : \mathcal{S} \rightarrow V_i$ be the corresponding projection operators. Similarly to the elastic isotropic case in section 2, 
the system (\ref{AE-E-L}) can be written as
\begin{equation}\label{AE_P}
  \begin{aligned}
    & P_1 \left( \Delta \Qvec_{ij} + \frac{L_2}{2} \Bigl( \Qvec_{ik,kj} + \Qvec_{jk,ki}  - \frac{2}{3} \delta_{ij} \Qvec_{kl, kl} \Bigr) \right)  \\
    &  ~~~  = t \Qvec_{ij} - 3 \sqrt{6} \Bigl( \Qvec_{ik}\Qvec_{kj} - \frac{1}{3} \delta_{ij} \tr(\Qvec^2) \Bigr) + 2 \Qvec_{ij} \tr(\Qvec^2), \\
    & P_2 \left(\Delta \Qvec_{ij} + \frac{L_2}{2} \Bigl( \Qvec_{ik,kj} + \Qvec_{jk,ki}  - \frac{2}{3} \delta_{ij} \Qvec_{kl, kl} \Bigr) \right) = 0, \\
    & P_3 \left(\Delta \Qvec_{ij} + \frac{L_2}{2} \Bigl( \Qvec_{ik,kj} + \Qvec_{jk,ki}  - \frac{2}{3} \delta_{ij} \Qvec_{kl, kl} \Bigr) \right) = 0.   \\
  \end{aligned}
\end{equation}


Direct calculations show that
\begin{equation}\label{Qikkj}
  \begin{aligned}
    \Qvec_{ik,kj}  & = \left( \n \otimes \n - \frac{1}{3} \mathbf{I} \right) (\nabla^2 s) + (\nabla s \cdot \n) \nabla \n \\
    & + \n \otimes \left( (\nabla \n)^{\mathrm{T}} \nabla s \right) + (\nabla \n) \n \otimes \nabla s + (\nabla \cdot \n) \n \otimes \nabla s \\
    & + s \Bigl( (\nabla^2 \n) \n + \nabla \n \nabla \n + (\nabla \cdot \n) \nabla \n + \n \otimes \nabla (\nabla \cdot \n)  \Bigr), \\
  \end{aligned}
\end{equation}
and 
\begin{equation}
  \begin{aligned}
    \frac{2}{3}  \Qvec_{kl, kl}  & = \frac{2}{3} \biggl( \pp_{kl}^2 s (n_k n_l  - \frac{1}{3} \delta_{kl}) + 2 \nabla s \cdot (\nabla \n) \n + 2 (\nabla \cdot \n) (\nabla s \cdot \n)  \\
    & + s \bigl( (\nabla \cdot \n )^2 + \tr(\nabla \n \nabla \n) + 2  \nabla (\nabla \cdot \n) \cdot \n \bigr) \biggr), \\
  \end{aligned}
\end{equation}
where $ \bigl( \nabla^2 s \bigr)_{ij} = \dfrac{\pp^2 s}{\pp x_i \pp x_j} = s_{ij}$, $ \bigl(\nabla \n \bigr)_{ij} = \dfrac{\pp n_i}{\pp x_j} = n_{i, j} $, $ \bigl( \nabla^2 \n \bigr)_{ijk} = \dfrac{\pp n_i}{\pp x_j \pp x_k} = n_{i, jk} $, $\bigl(\nabla \n \nabla \n \bigr)_{ij} = n_{i, k} n_{k, j}$, 
and $ \bigl( (\nabla^2 \n) \n \bigr)_{ij} = n_{i, jk} n_k$.

It can be noticed that
\begin{equation}
 \begin{aligned}
& (\nabla \n) \n = (\nabla \n) \n - (\nabla \n)^{\mathrm{T}} \n = - \n \times (\nabla \times \n) \in \n^{\perp},  \\
& (\n \otimes \n) \nabla^2 s =  n_i n_k s_{kj} = n_i s_{jk} n_k =  \n \otimes \left( (\nabla^2 s ) \n \right). \\
\end{aligned} 
\end{equation}


For $\forall \vvec \in \mathbb{R}^3$ and $\forall \w \in \n^{\perp}$, we have
\begin{equation}\label{st_nm}
\begin{aligned}
& \mathcal{ST} \left( \n \otimes \n \right) = \n \odot \n - \frac{1}{3} \mathbf{I}, \\
& \mathcal{ST} (\n \otimes \vvec) = (\vvec \cdot \n) \left( \n \odot \n - \frac{1}{3} \mathbf{I} \right) + \n \odot \left(\vvec - (\vvec \cdot \n) \n \right), \\ 
& \mathcal{ST} (\w \otimes \vvec) = \n \odot \left( (\vvec \cdot \n) \w \right) + \mathcal{ST} \left( \w \odot (\vvec - (\vvec \cdot \n) \n )  \right), \\
\end{aligned}
\end{equation}
where $\odot$ denotes the symmetric tensor product $(\n \odot \m)_{ij} =  \frac{1}{2}(n_i m_j + n_j m_i)$, and $\mathcal{ST} \left( \mathbf{A} \right)$ is the symmetric, traceless part of a matrix $\mathbf{A}$, i.e. $\mathcal{ST}(\mathbf{A}) = \frac{1}{2}( \mathbf{A} + \mathbf{A}^{\mathrm{T}}) - \frac{1}{3} \tr( \mathbf{A} ) \mathbf{I}, \quad \forall \mathbf{A} \in \mathbb{R}^{3 \times 3}$. Hence, from (\ref{Qikkj}), we have
\begin{equation}\label{AE_Q_term}
  \begin{aligned}
    & \frac{1}{2} \Bigl( \Qvec_{ik,kj}  + \Qvec_{jk,ki} - \frac{2}{3} \delta_{ij} \Qvec_{kl, kl} \Bigr) = \mathcal{ST} \left( \Qvec_{ik, kj} \right) \\
    & =  \left( (\nabla \cdot \n) (\nabla s \cdot \n) + \nabla s \cdot (\nabla \n) \n + s \nabla (\nabla \cdot \n) \cdot \n + (\nabla^2 s) \n \cdot \n \right) \left(\n \odot  \n - \frac{1}{3} \mathbf{I} \right)\\
    & + \n \odot \Bigl( \left( (\nabla^2 s) \n - ((\nabla^2 s) \n \cdot \n) \n \right) + \left( (\nabla \n)^{\mathrm{T}} \nabla s - (\nabla s \cdot (\nabla \n) \n) \n \right) + (\nabla s \cdot \n)  (\nabla \n) \n  \\
    & ~~~~~~~~~~~ + (\nabla \cdot \n) \left( \nabla s - (\nabla s \cdot \n) \n \right) + s \left( \nabla (\nabla \cdot \n) -  \left( \nabla (\nabla \cdot \n) \cdot \n \right) \n \right) \Bigr) + R(s, \n), \\
  \end{aligned}
\end{equation}
 where
\begin{equation}
  \begin{aligned}
 R(s, \n) & = \mathcal{ST}  \bigl( (\nabla \n) \n \odot (\nabla s - (\nabla s \cdot \n)  \n ) \bigr)  + (\nabla s \cdot \n + s \nabla \cdot \n) \mathcal{ST}(\nabla \n)  \\
      & + s~ \mathcal{ST} \bigl( \nabla \n \nabla \n + (\nabla^2 \n)\n \bigr) - \frac{1}{3} \left( \nabla^2 s - \frac{1}{3} (\Delta s) \mathbf{I} \right). \\
  \end{aligned}
\end{equation}
The detailed calculations leading to (\ref{AE_Q_term}) are given in the Appendix.


Unlike the elastic isotropic case, we are unable to get explicit equations for $s$ and $\n$, as the projections of $R(s, \n)$ depend on $s$ and $\n$.
Moreover, according to (\ref{AE_Q_term}), all the equations in (\ref{AE_P}) involve the second derivatives of $\n$ and $s$. Hence, the uniaxial assumption gives stronger constraints in the elastic anisotropic case compared to the elastic isotropic case.
We consider uniaxial solutions with certain symmetries below.

\begin{proposition}
\label{prop:4.1} If 
\begin{equation}\label{Ansatz}
\Qvec(r, \theta, \varphi) = s(r) \left( \n(\theta, \varphi) \otimes \n(\theta, \varphi) - \frac{1}{3} \mathbf{I} \right)
\end{equation}
is a non-trivial uniaxial solution of (\ref{AE-E-L}),
then 
\begin{equation}
\n(\theta, \varphi) = \frac{\x}{|\x|}
\end{equation}
and $s$ is a solution of 
\begin{equation}\label{AE_eq_h}
  \left( 1 + \frac{2}{3} L_2 \right) \Bigl( s''(r) + \frac{2}{r} s'(r) \Bigr) = \left( 1 + \frac{2}{3} L_2 \right) \frac{6}{r^2} s(r) + \psi(s(r)),  
\end{equation}
where $\psi(s) = ts - \sqrt{6} s^2 + \dfrac{4}{3} s^3$.

\end{proposition}

\begin{proof}

Let
\begin{equation}
  \begin{aligned}
    & \mathbf{e}_r = \left( \sin \varphi \cos \theta,~ \sin \varphi \sin \theta,~ \cos \varphi \right), \quad \mathbf{e}_{\varphi} = \left( \cos \varphi \cos \theta,~ \cos \varphi \sin \theta,~ - \sin \varphi \right), \\
    &  \mathbf{e}_{\theta} = \left( - \sin \theta,~ \cos \theta,~ 0 \right), \quad \n = \left( \sin f  \cos g, ~ \sin f   \sin g,~ \cos f \right), \\
    &  \m = \left( \cos f \cos g,~ \cos f \sin g,~ - \sin f \right), \quad \p = \left( - \sin g,~ \cos g,~ 0 \right).
  \end{aligned}
\end{equation}

Thus,
\begin{equation}
\begin{aligned}
& V_1 = \mathrm{span} \left\{ \n \odot \n - \frac{1}{3} \mathbf{I} \right\}, \\
& V_2 = \mathrm{span} \left\{ \n \odot \m, ~~  \n \odot \p \right\}, \\
& V_3 = \mathrm{span} \left\{ \m \odot  \m - \p \odot \p, ~~ \m \odot \p \right \}. \\
\end{aligned}
\end{equation}

Then the system (\ref{eq_dimless_AE}) can be written as
\begin{equation}
\begin{aligned}
& K_1 (s, \n) \left( \n \odot \n - \frac{1}{3} \mathbf{I} \right) +  K_2 (s, \n) \left( \n \odot \m \right) + K_3 (s, \n) \left( \n \odot \p \right)  \\
& + K_4 (s, \n) \left( \m \odot  \m - \p \odot \p \right) + K_5 (s, \n) \left( \m \odot \p \right) = 0, \\
\end{aligned}
\end{equation}
which gives us five equations for $s$ and $\n$, i.e. $K_i (s, \n) = 0,  i = 1, \ldots 5$.

For clarity of presentation, we consider the special case for which
\begin{equation}\label{AE_n_sp}
    \n(\theta, \varphi) = (\sin f(\varphi) \cos g (\theta), \sin f(\varphi) \sin g(\theta),  \cos f(\varphi) ). 
\end{equation}

Since $s = s(r)$, we have
\begin{equation}
 \nabla s = \pp_rs ~ \e_r,  \quad  \nabla^2 s = \pp_r^2 s ~ \e_r \otimes \e_r + \frac{1}{r} \Bigl( \pp_r s ~ (\e_{\varphi} \otimes \e_{\varphi} + \e_{\theta} \otimes \e_{\theta}) \Bigr),
\end{equation}
and
\begin{equation}\label{dds}
\nabla^2 s  - \frac{1}{3}(\Delta s) \mathbf{I} = (\pp_r^2 s - \frac{1}{r} \pp_r s) \Bigl( \e_r \otimes \e_r - \frac{1}{3} \mathbf{I} \Bigr). 
\end{equation}

For $\n$ of the form (\ref{AE_n_sp}), direct calculations show that
\begin{equation}\label{Cal_sn1}
\begin{aligned}
& \nabla \cdot \n  = \frac{1}{r} \left( (\m, \e_{\varphi}) \pp_{\varphi} f + \frac{\sin f}{\sin \varphi} (\p, \e _{\theta}) \pp_{\theta} g \right) \triangleq \frac{1}{r} D(\theta, \varphi), \\
& \nabla \n = \frac{1}{r} \left(  \partial_{\varphi} f ~ \m \otimes \e_{\varphi} + \frac{\sin f}{\sin \varphi}   \partial_{\theta} g ~ \p \otimes \e_{\theta} \right),  \\
& (\nabla \n) \n  = \frac{1}{r} \left( \pp_{\varphi} f (\n, \e_{\varphi}) \m + \frac{\sin f}{ \sin \varphi} \pp_{\theta} g  (\n, \e_{\theta}) \p \right), \\
& \nabla s - \left( \nabla s \cdot \n \right) \n  =  \pp_r s \Bigl( (\m, \e_r) \m + (\p, \e_r) \p \Bigr), \\ 
\end{aligned}
\end{equation}
where $(\cdot, \cdot)$ is the inner product in $\mathbb{R}^3$.

Hence,
\begin{equation}\label{Cal_sn2}
  \begin{aligned}
 &  (\nabla \cdot \n)(\nabla s \cdot \n) + \nabla s \cdot (\nabla \n)\n  \\
    & ~~~~ = \left( \Bigl((\m, \e_{\varphi})(\n, \e_r) + (\m, \e_r)(\n, \e_{\varphi}) \Bigr) \pp_{\varphi} f + \frac{\sin f}{\sin \varphi} \Bigl( (\p, \e_{\theta})(\n, \e_r) + (\p, \e_r)(\n, \e_{\theta})\Bigr)  \pp_{\theta} g \right) \dfrac{1}{r} \pp_r s, \\
& (\nabla^2 s) \n \cdot \n = (\n, \e_r)^2 \pp_r^2 s + \left( (\n, \e_{\varphi})^2 + (\n, \e_{\theta})^2 \right)\frac{1}{r} \pp_r s = (\n, \e_r)^2 \pp_r^2 s + \left(1 - (\n, \e_{r})^2 \right)\frac{1}{r} \pp_r s, \\
\end{aligned}
\end{equation}
and
\begin{equation}\label{Cal_sn3}
  \begin{aligned}
    & \nabla (\nabla \cdot \n ) = \pp_r(\nabla \cdot \n) \e_r + \dfrac{1}{r} \pp_{\varphi} (\nabla \cdot \n) \e_{\varphi} + \dfrac{1}{r \sin \varphi} \pp_{\theta} (\nabla \cdot \n) \e_{\theta} \\
    & ~~~~~~~~~~~ = \dfrac{1}{r^2} \left( -  D(\theta, \varphi) \e_r + \pp_{\varphi}  D(\theta, \varphi) \e_{\varphi} + \dfrac{1}{\sin \varphi}  \pp_{\theta} D(\theta, \varphi) \e_{\theta}    \right), \\
    & \nabla \n \nabla \n = \frac{1}{r^2} \biggl( (\pp_{\varphi} f)^2 (\m, \e_{\varphi}) \m \otimes \e_{\varphi}  \\
& ~~~~~~~~~~~~~~~~~~~ + \frac{\sin^2 f}{\sin^2 \varphi} (\pp_{\theta} g)^2  (\p, \e_{\theta}) \p \otimes \e_{\theta}  + \frac{\sin f}{\sin \varphi} \pp_{\varphi} f \pp_{\theta} g \Bigl( (\p, \e_{\varphi})\m \otimes \e_{\theta} + (\m, \e_{\theta}) \p \otimes \e_{\varphi} \Bigr) \biggr), \\
    &  \mathcal{S}(\nabla \n ) =  \frac{1}{r} \biggl( \pp_{\varphi} f (\m, \e_{\varphi}) \m \odot \m  + \frac{\sin f}{\sin \varphi} \pp_{\theta} g  (\p, \e_{\theta}) \p \odot \p + \Bigl( \pp_{\varphi} f (\p, \e_{\varphi}) +  \frac{\sin f}{\sin \varphi} \pp_{\theta} g  (\m, \e_{\theta}) \Bigr) \p \odot \m \biggr) \\
& \qquad ~~~~~ + \n \odot  \frac{1}{r} \left( \pp_{\varphi} f (\m, \e_{\varphi}) \m + \dfrac{\sin f}{\sin \varphi} \pp_{\theta} g  (\n, \e_{\theta}) \p \right),  \\ 
 \end{aligned}
\end{equation}
 where $\mathcal{S} \left( \mathbf{A} \right)$ is the symmetric part of a matrix $\mathbf{A}$, i.e. $\mathcal{S}(\mathbf{A}) = \frac{1}{2}( \mathbf{A} + \mathbf{A}^{\mathrm{T}}), \quad \forall \mathbf{A} \in \mathbb{R}^{3 \times 3}$.

Next, we compute $\mathcal{S} ((\nabla^2 \n) \n)$. Since
\begin{equation}
(\nabla^2 \n) \n  = (\n, \e_r)  \pp_r(\nabla \n) + \frac{1}{r} (\n, \e_{\varphi})  \pp_{\varphi}(\nabla \n) + \frac{1}{r \sin \varphi}(\n, \e_{\theta})  \pp_{\theta}(\nabla \n),
\end{equation}
where
\begin{equation}
\begin{aligned}
& \pp_r (\nabla \n) = - \frac{1}{r^2} \left( \partial_{\varphi} f ~ \m \otimes \e_{\varphi} + \frac{\sin f}{\sin \varphi}   \partial_{\theta} g ~ \p \otimes \e_{\theta} \right), \\
& \pp_{\varphi} (\nabla \n)  
    = \frac{1}{r} \left( \partial_{\varphi}^2 f ~ \m \otimes \e_{\varphi} +  \pp_{\varphi} \left(\frac{\sin f}{\sin \varphi}\right) \partial_{\theta} g ~ \p \otimes \e_{\theta} - (\pp_{\varphi} f)^2 \n \otimes \e_{\varphi} - \pp_{\varphi} f ~ \m \otimes \e_r  \right), \\
    & \pp_{\theta} (\nabla \n) 
    =  \frac{1}{r} \Bigl( \pp_{\varphi} f ( \cos f \pp_{\theta} g~ \p \otimes \e_{\varphi} + \cos \varphi ~ \m \otimes \e_{\theta}) \\
    & + \frac{\sin f}{\sin \varphi} \left( \pp_{\theta}^2 g ~ \p \otimes \e_{\theta} -  (\pp_{\theta} g)^2~  (\sin f ~ \n + \cos f ~ \m) \otimes \e_{\theta} -  \pp_{\theta} g~  \p \otimes (\sin \varphi ~ \e_r + \cos \varphi~ \e_{\varphi}) \right) \Bigr). \\
    \end{aligned}
\end{equation}
We have
\begin{equation}\label{Cal_sn4}
  \begin{aligned}
    \mathcal{S} ((\nabla^2 \n)) \n  & = \frac{1}{r^2} \Biggl( \left( (\n, \e_{\varphi}) \pp_{\varphi}^2 f - (\n, \e_r) \pp_{\varphi} f \right) \m \odot \e_{\varphi} \\
                                    & \qquad + \frac{1}{\sin \varphi} (\n, \e_{\theta}) \left(\cos \varphi \pp_{\varphi} f - \frac{\sin f}{\sin \varphi} \cos f (\pp_{\theta} g)^2 \right) \m \odot \e_{\theta} \\
    & \qquad + \left( \frac{\sin f}{\sin^2 \varphi} (\n, \e_{\theta}) \pp_{\theta}^2 g + \pp_{\varphi}(\frac{\sin f}{\sin \varphi}) (\n, \e_{\varphi}) \pp_{\theta} g - \frac{\sin f}{\sin \varphi} (\n, \e_r) \pp_{\theta} g ) \right) \p \odot \e _{\theta} \\
    & \qquad + \frac{1}{ \sin \varphi} (\n, \e_{\theta}) \left( \cos f \pp_{\varphi} f \pp_{\theta} g - \frac{\sin f}{\sin \varphi} \cos \varphi \pp_{\theta} g \right) \p \odot \e_{\varphi} \\
    & \qquad - (\n, \e_{\varphi}) (\pp_{\varphi} f)^2 \n \odot \e_{\varphi} - \frac{1}{r^2} (\n, \e_{\varphi}) \pp_{\varphi} f ~ \m \odot \e_r \\
    & \qquad - \frac{\sin f}{\sin^2 \varphi} (\n, \e_{\theta}) \sin f (\pp_{\theta}g)^2~ \n \odot \e_{\theta} - \frac{\sin f}{\sin^2 \varphi} (\n, \e_{\theta}) \sin \varphi \pp_{\theta} g ~ \p \odot \e_r \Biggr).
    \end{aligned}
\end{equation}

In order to get $K_1(s, \n)$, we need to project $R(s, \n)$ into $V_1$. Note
\begin{equation}
\mathcal{ST} \left( \mu_1 (\m \odot \m) + \mu_2 (\p \odot \p) \right) =  \dfrac{\mu_1 - \mu_2}{2} \left( \m \odot \m - \p \odot \p \right) - \dfrac{\mu_1 + \mu_2}{2} \left( \n \odot \n - \dfrac{1}{3} \mathbf{I} \right)
\end{equation}
for $\forall \mu_1, \mu_2 \in \mathbb{R}$. Hence,
\begin{equation}
\begin{aligned}
& P_1 \left( \nabla^2 s - \frac{1}{3} (\Delta s) \mathbf{I} \right) = \left( \frac{3}{2} (\n, \e_r)^2 - \frac{1}{2} \right) \left( \pp_r^2 s - \frac{1}{r} \pp_r s \right)  \Bigl( \n \odot \n - \frac{1}{3} \mathbf{I} \Bigr), \\
&  P_1 \Bigl( \mathcal{ST}  \bigl( (\nabla \n) \n \odot (\nabla s - (\nabla s \cdot \n)  \n) \bigr)  + (\nabla s \cdot \n) \mathcal{ST}(\nabla \n) \Bigr)
=  B_0 (\theta, \varphi) \frac{1}{r} \pp_r s(r) \Bigl( \n \odot \n - \frac{1}{3} \mathbf{I} \Bigr), \\
& P_1  \Bigl( \left( \nabla \cdot \n \right) ~ \mathcal{ST} (\nabla \n) + \mathcal{ST}\left( \nabla \n \nabla \n + (\nabla^2 \n)\n \right) \Bigr) =  \frac{1}{r^2}C_0(\theta, \varphi) \Bigl( \n \odot \n - \frac{1}{3} \mathbf{I} \Bigr), \\
\end{aligned}
\end{equation}
where $B_0 (\theta, \varphi), C_0(\theta, \varphi)$ depend on $f$ and $g$, which can be calculated from (\ref{Cal_sn3}), (\ref{Cal_sn4}). 
One can show that
\begin{equation}
B_0(\theta, \varphi) = - \dfrac{1}{2}\left( \Bigl( (\n, \e_{\varphi})(\m, \e_r) + (\m, \e_{\varphi})(\n, \e_r) \Bigr) \pp_{\varphi} f  + \frac{\sin f}{\sin \varphi} \Bigl( (\n, \e_{\theta})(\p, \e_r) + (\p, \e_{\theta})(\n, \e_r)  \Bigr) \pp_{\theta} g \right).
\end{equation}
The expression of $C_0(\theta, \varphi)$ is rather complicated and does not play any role in our proof.

The above calculations imply that $K_1 (s, \n) = 0$ is equivalent to
\begin{equation}\label{AE_P1_s}
A_1(\theta, \varphi) s''(r) + B_1(\theta, \varphi) \frac{1}{r} s'(r) + C_1(\theta, \varphi) \frac{1}{r^2} s(r) = \psi (s),
\end{equation}
where
\begin{equation}\label{AE_P1_AB}
\begin{aligned}
& A_1(\theta, \varphi) s''(r) + B_1(\theta, \varphi) \frac{1}{r} s'(r) =  \Delta s + L_2 \left( (\nabla \cdot \n) (\nabla s \cdot \n) + \nabla s \cdot (\nabla \n) \n + (\nabla^2 s) \n \cdot \n \right)  \\
&  ~~ + L_2 B_0 (\theta, \varphi)  \frac{1}{r} s'(r) - L_2 \left( \dfrac{1}{2} (\n, \e_r)^2 - \dfrac{1}{6} \right) \left( s''(r) - \frac{1}{r} s'(r) \right),
\end{aligned}
\end{equation}
and
\begin{equation}
 \frac{1}{r^2}C_1(\theta, \varphi) =  L_2 \left( \frac{1}{r^2}C_0(\theta, \varphi) + \nabla (\nabla \cdot \n) \cdot \n \right) - 3 |\nabla \n|^2.
\end{equation}

From (\ref{Cal_sn2}), we have
\begin{equation}
\begin{aligned}
& (\nabla \cdot \n) (\nabla s \cdot \n) + \nabla s \cdot (\nabla \n) \n + (\nabla^2 s) \n \cdot \n  = (\n, \e_r)^2 s''(r) \\
& ~~+ \biggl(\Bigl( (\n, \e_{\varphi})(\m, \e_r) + (\m, \e_{\varphi})(\n, \e_r) \Bigr) \pp_{\varphi} f + \frac{\sin f}{\sin \varphi} \Bigl( (\n, \e_{\theta})(\p, \e_r) + (\p, \e_{\theta})(\n, \e_r) \Bigr)  \pp_{\theta} g \\
& \quad \quad \quad \quad + 1 - (\n, \e_r)^2  \biggr)\dfrac{1}{r} s'(r). \\
\end{aligned}
\end{equation}
Hence, (\ref{AE_P1_AB}) implies that 
\begin{equation}\label{AE_A1_value}
\begin{aligned}
 A_1 (\theta, \varphi) & = 1 + L_2 \left( (\n, \e_r)^2 - \left(  \dfrac{1}{2} (\n, \e_r)^2 - \dfrac{1}{6}  \right) \right) \\
& = 1 + L_2 \left( \frac{1}{2} (\n, \e_r)^2 + \frac{1}{6} \right) \neq 0, 
\end{aligned}
\end{equation}
and
\begin{equation}
\begin{aligned}
  & B_1(\theta, \varphi)  = 2 + L_2 \biggl( \frac{5}{6} - \frac{1}{2} (\n, \e_r)^2 + \frac{1}{2} \Bigl( (\n, \e_{\varphi})(\m, \e_r) + (\m, \e_{\varphi})(\n, \e_r) \Bigr) \pp_{\varphi} f  \\
  & ~~~~~~~~~~~~~~~~~~~~~~~~~  + \frac{1}{2} \frac{\sin f}{\sin \varphi} \Bigl( (\n, \e_{\theta})(\p, \e_r) + (\p, \e_{\theta})(\n, \e_r) \Bigr)  \pp_{\theta} g  \biggr). \\
\end{aligned}
\end{equation}


Similarly, from (\ref{AE_Q_term}), one can show that
\begin{equation}\label{AE_extra}
\begin{aligned}
  & P_3 \left(\Delta \Qvec_{ij} + \frac{L_2}{2} \Bigl( \Qvec_{ik,kj} + \Qvec_{jk,ki}  - \frac{2}{3} \delta_{ij} \Qvec_{kl, kl} \Bigr) \right) = P_3\Bigl( \Delta \Qvec_{ij} + L_2 R(s, \n) \Bigr) \\
  & \qquad \qquad = s''(r) \mathbf{A} (\theta, \varphi)  +  \frac{1}{r}s'(r) \mathbf{B} (\theta, \varphi) +  \frac{1}{r^2}s(r) \mathbf{C} (\theta, \varphi) = 0, \\
  \end{aligned}
\end{equation}
where
\begin{equation}\label{AE_Eaxta_Coe_AB}
\begin{aligned}
& s''(r) \mathbf{A} (\theta, \varphi)  +  \frac{1}{r}s'(r) \mathbf{B} (\theta, \varphi) \\
&  = L_2 P_3 \left( \mathcal{ST} \bigl( (\nabla \n) \n \odot (\nabla s - (\nabla s \cdot \n) \n ) \bigr) + (\nabla s \cdot \n) \mathcal{ST}(\nabla \n) - \dfrac{1}{3} \Bigl( \nabla^2 s - \dfrac{1}{3} (\Delta s) \mathbf{I} \Bigr)  \right) \\
&  = \left( A_4(\theta, \varphi)s''(r) + B_4(\theta, \varphi) \dfrac{1}{r} s'(r)  \right) \left( \m \odot \m - \p \odot \p \right)  + \left( A_5(\theta, \varphi)s''(r) + B_5(\theta, \varphi) \dfrac{1}{r} s'(r)  \right) \m \odot \p, \\
\end{aligned}
\end{equation}
and
\begin{equation}\label{AE_Eaxta_Coe_C}
\begin{aligned}
 \frac{1}{r^2} \mathbf{C} (\theta, \varphi)  & = \left( 2 \sum_{k=1}^3 \pp_k \n \otimes \pp_k \n  - |\nabla \n|^2( \mathbf{I} - \n \otimes \n) \right) + L_2  \biggl( P_3 \Bigl(  \mathcal{ST} \bigl( (\nabla \cdot \n)\nabla \n + \nabla \n \nabla \n + (\nabla^2 \n)\n \bigr) \Bigr) \biggr), \\ 
 & =  \frac{1}{r^2} \Bigl( C_4(\theta, \varphi) \left( \m \odot \m - \p \odot \p \right)  + C_5(\theta, \varphi)\m \odot \p \Bigr). \\ 
\end{aligned}
\end{equation}
Hence, 
\begin{equation}\label{AE_Extra_2}
K_i (s, \n) = 0 ~~\Longleftrightarrow~~ A_i(\theta, \varphi) s''(r) + B_i(\theta, \varphi) \frac{1}{r} s'(r) + C_i(\theta, \varphi) \frac{1}{r^2} s(r) = 0, \quad i = 4, 5.
\end{equation}

From (\ref{AE_Eaxta_Coe_AB}) and (\ref{dds}), we have
\begin{equation}
\begin{aligned}
 \mathbf{A} (\theta, \varphi) &  = - \frac{L_2}{3} P_3 \Bigl( \e_r \otimes \e_r - \frac{1}{3} \mathbf{I} \Bigr) \\
                              & = - \frac{L_2}{3} \left( \frac{1}{2} \left( (\m, \e_r)^2 - (\p, \e_r)^2 \right)     \left( \m \odot \m - \p \odot \p \right)  + 2 (\m, \e_r) (\p, \e_r)\m \odot \p \right). \\
\end{aligned}
\end{equation}
Hence,
\begin{equation}\label{Extra_A45}
A_4(\theta, \varphi) = -\dfrac{L_2}{6} \left( (\m, \e_r)^2 - (\p, \e_r)^2 \right), \quad A_5(\theta, \varphi) = - \dfrac{2L_2}{3} (\m, \e_r)(\p, \e_r).
\end{equation}


The two equations in (\ref{AE_Extra_2}) can be viewed as two linear ordinary differential equations for $s(r)$. 
If $\exists k \in \{ 4, 5 \}$, s.t. $A_k(\theta, \varphi) \neq 0$, we can obtain $s(r)$ by solving the equation in (\ref{AE_Extra_2}) with $A_k \neq 0$, which cannot be a solution of (\ref{AE_P1_s}).
Indeed, solutions of the equation in (\ref{AE_Extra_2}) with $A_k \neq 0$, are of the form
\begin{equation}\label{sol_extra}
 \gamma_1 r^{\alpha_1} + \gamma_2 r^{\alpha_2}, ~~~\text{or}~~~  (\gamma_1 + \gamma_2 \ln r) r^{\alpha_1}, ~~~\text{or} ~~~ r^{\alpha_1}\left( \gamma_1 \cos(\alpha_2 \ln r) + \gamma_2 \sin(\alpha_2 \ln r) \right),
\end{equation} 
depending on $A_k$, $B_k$, and $C_k$ \cite{ODEBook}. However, the solutions in (\ref{sol_extra}) cannot be solutions of (\ref{AE_P1_s}). 


So $A_4(\theta, \varphi) = A_5(\theta, \varphi) = 0$, which implies that $(\m, \e_r) = (\p, \e_r) = 0$. Since $\n$, $\m$ and $\p$ are pairwise orthogonal, we have $\n = \e_r =\dfrac{\x}{|\x|}$.

For $\n = \dfrac{\x}{|\x|}$, direct calculations show that
\begin{equation}
P_i \left( \Delta \Qvec_{ij} + \frac{L_2}{2} \Bigl( \Qvec_{ik,kj} + \Qvec_{jk,ki}  - \frac{2}{3} \delta_{ij} \Qvec_{kl, kl} \Bigr) \right) = 0,~ i = 2, 3,
\end{equation} 
and $s$ is a solution of
\begin{equation}
  \left( 1 + \frac{2}{3} L_2 \right) \Bigl( s''(r) + \frac{2}{r} s'(r) \Bigr) = \left( 1 + \frac{2}{3} L_2 \right) \frac{6}{r^2} s(r) + \psi(s(r)),
\end{equation}
where $\psi(s) = ts - \sqrt{6} s^2 + \dfrac{4}{3} s^3$.

For a general $\n(\theta, \varphi)$, we note that $\n = \n(\theta, \varphi)$ implies that
\begin{equation}
\dfrac{\pp n_i}{\pp x_j} = O \left( \frac{1}{r} \right), \quad \dfrac{\pp^2 n_i }{\pp x_j \pp x_k} = O \left( \dfrac{1}{r^2} \right).
\end{equation}
Hence, as in the special case, 
\begin{equation}\label{eq_K1}
K_1 (s, \n) = 0 ~~\Longleftrightarrow~~ A_1(\theta, \varphi) s''(r) + B_1(\theta, \varphi) \frac{1}{r} s'(r) + C_1(\theta, \varphi) \frac{1}{r^2} s(r) = \psi (s(r)), 
\end{equation}
and
\begin{equation}\label{eq_Ki}
K_i (s, \n) = 0 ~~\Longleftrightarrow~~ A_i(\theta, \varphi) s''(r) + B_i(\theta, \varphi) \frac{1}{r} s'(r) + C_i(\theta, \varphi) \frac{1}{r^2} s(r) = 0, \quad i = 2, 3, 4, 5.
\end{equation}

We can conclude the proof by noting that (\ref{AE_A1_value}) and (\ref{Extra_A45})  always hold, as $A_1$, $A_4$ and $A_5$ are all determined by $\nabla^2 s$.
Hence, if 
\begin{equation}
\Qvec(r, \theta, \varphi) = s(r) \left( \n(\theta, \varphi) \otimes \n (\theta, \varphi) - \frac{1}{3} \mathrm{I} \right)
\end{equation}  
is a non-trivial uniaxial solution of (\ref{eq_dimless_AE}), then $\n = \dfrac{\x}{|\x|}$ and $s$ is a solution of
\begin{equation}
\left( 1 + \frac{2}{3} L_2 \right) \Bigl( s''(r) + \frac{2}{r} s'(r) \Bigr) = \left( 1 + \frac{2}{3} L_2 \right)\frac{6}{r^2} s(r) + \psi(s(r)),
\end{equation}
where $\psi(s) = ts - \sqrt{6} s^2 + \dfrac{4}{3} s^3$.

\end{proof}

\section{Conclusions}\label{sec:conclusions}

We study uniaxial solutions for the Euler-Lagrange equations in the LdG framework, to some extent building on the results in \cite{lamy2015uniaxial}. There is existing work on the uniaxial/biaxial character of LdG equilibria, they rely on energy comparison arguments and the fact that biaxiality is preferred at low temperatures, to the uniaxial phase, or that biaxiality arises from geometrical considerations. We purely use the structure of the Euler-Lagrange equations (as in \cite{lamy2015uniaxial}) in this framework, and our results therefore apply to all critical points and not merely minimizers.

For a 3D problem, a uniaxial LdG $\Qvec$-tensor has three degrees of freedom whereas a fully biaxial tensor has five degrees of freedom. By using spherical angles to represent the unit vector $\n = (\sin f \cos g, \sin f \sin g, \cos f)$, we derive a system of partial differential equations for $f$, $g$ and the scalar order parameter, $s$. We believe that this representation of uniaxial solutions will aid further work in this direction.

In the elastic isotropic case, under the assumption that  $f = f(\varphi)$ and $g = g(\theta)$, we show that the only possible uniaxial solutions are $f(\varphi) = \pm \varphi$, $g(\theta) = \pm \theta + C$, and $s = s(r)$ satisfies a second order ordinary differential equation.
In other words, they are radial-hedgehog solutions modulo an orthogonal transformation. 
By using an orthonormal basis for the space of symmetric and traceless tensors, we can show that if $\e_z$ is a eigenvector of $\Qvec$, 
then $\Qvec$ necessarily has a constant eigenframe. 

In the elastic anisotropic case, we can show the radial-hedgehog is the only possible uniaxial solution under the assumption that $s = s(r)$, $f = f(\theta, \varphi)$ and $g = g(\theta, \varphi)$. Although a complete description of 3D uniaxial solutions is still missing, we believe the radial-hedgehog is the only nontrivial uniaxial solution, at least in the elastic anisotropic case. Further, we consider model problems in this paper but these model problems are physically relevant, e.g. it is reasonable to expect that the uniaxial director is independent of $r$ for spherically symmetric geometries or that $\mathbf{e}_z$ is a fixed eigenvector for severely confined systems, with $\mathbf{e}_z$ normal to the bounding plates. The formulation of the uniaxial problem in terms of $s$, $f$ and $g$ will be useful for a completely general study of admissible uniaxial solutions of the LdG Euler-Lagrange equations without any constraints.

Pure uniaxiality appears to be a strong constraint but it is known that for several model situations, (see e.g. \cite{amaz, henaomajumdarpisante2017}), minimizers are approximately uniaxial almost everywhere. Therefore, it would be interesting and highly instructive to construct ``explicit'' approximately uniaxial solutions. Our technical computations in the elastic isotropic and anisotropic case may aid such constructions and equally, similar techniques may help in classifying solutions (without the constraint of uniaxiality) of the LdG Euler-Lagrange equations.

\section*{Acknowlegements}

A.M.'s research is supported by an EPSRC Career Acceleration Fellowship EP/J001686/1 and EP/J001686/2, an OCIAM Visiting Fellowship and the Advanced Studies Centre at Keble College.
Part of this work was carried out when Y.W. was visiting the University of Bath, he would like to thank the University of Bath and Keble College for their hospitality.
He also would like to thank the National Natural Science Foundation of China for financial support (grant No. 11421101) and his Ph.D. advisor Pingwen Zhang, for his constant support and helpful advice.
\numberwithin{equation}{section}


\numberwithin{equation}{section}

\begin{appendices}

\section{Calculations of Eq. (\ref{AE_Q_term})}
In order to get (\ref{AE_Q_term}), we compute the symmetric, traceless part of each term in (\ref{Qikkj}), the first step of which is eq. (\ref{st_nm}). Note that
\begin{equation}
 ( (\nabla \n)^{\mathrm{T}} \nabla s ) \cdot \n =  (\nabla s)^{\mathrm{T}} (\nabla \n) \n = \nabla s \cdot (\nabla \n) \n.
\end{equation}

The direct calculations show that
\begin{equation}
\begin{aligned}
  & \mathcal{ST} \left( \left( \n \otimes \n - \frac{1}{3} \mathbf{I} \right) (\nabla^2 s) \right) = \mathcal{ST} \left( \n \otimes (\nabla^2 s)\n \right) - \mathcal{ST} \left( \dfrac{1}{3} \nabla^2 s \right)\\ 
  & \quad =  \Bigl(  (\nabla^2 s) \n \cdot \n \Bigr)  \left( \n \odot \n - \frac{1}{3} \mathbf{I} \right)  + \n \odot \Bigl(  (\nabla^2 s) \n - \left( (\nabla^2 s) \n \cdot \n \right) \n \Bigr) - \frac{1}{3} \left( \nabla^2 s - \frac{1}{3} (\Delta s) \mathbf{I} \right), \\
& \mathcal{ST} \Bigl( (\nabla s \cdot \n) \nabla \n \Bigr) = (\nabla s \cdot \n) \mathcal{ST} \Bigl( \nabla \n \Bigr), \\
& \mathcal{ST} \Bigl( \n \otimes \left( (\nabla \n)^{\mathrm{T}} \nabla s  \right) \Bigr) =  \left( ( (\nabla \n)^{\mathrm{T}} \nabla s ) \cdot \n \right) \left( \n \odot \n - \frac{1}{3} \mathbf{I} \right) + \n \odot \Bigl( (\nabla \n)^{\mathrm{T}} \nabla s - (  ( (\nabla \n)^{\mathrm{T}} \nabla s ) \cdot \n) \n \Bigr) \\
& \quad  = \left( \nabla s \cdot (\nabla \n) \n \right) \left( \n \odot \n - \frac{1}{3} \mathbf{I} \right) + \n \odot \Bigl( (\nabla \n)^{\mathrm{T}} \nabla s - (\nabla s \cdot (\nabla \n) \n) \n \Bigr), \\
& \mathcal{ST} \Bigl( (\nabla \n) \n \otimes \nabla s \Bigr) = (\nabla \n) \n \odot (\nabla s \cdot \n) \n  + \mathcal{ST} \Bigl( (\nabla \n) \n \odot ( \nabla s - (\nabla s \cdot \n) \n ) \Bigr) \\
   & \quad = \n \odot \Bigl( (\nabla s \cdot \n) (\nabla \n) \n \Bigr)  + \mathcal{ST} \Bigl( (\nabla \n) \n \odot ( \nabla s - (\nabla s \cdot \n) \n ) \Bigr), \\
& \mathcal{ST} \Bigl( (\nabla \cdot \n) \n \otimes \nabla s \Bigr) =  (\nabla \cdot \n) \mathcal{ST} \Bigl( \n \otimes \nabla s  \Bigr) \\
& \quad = (\nabla \cdot \n) (\nabla s \cdot \n) \left( \n \odot \n - \frac{1}{3} \mathbf{I} \right) + \n \odot \Bigl( (\nabla \cdot \n) \left( \nabla s - (\nabla s \cdot \n)\n \right) \Bigr), \\
 & \mathcal{ST} \left( s \Bigl( (\nabla^2 \n) \n + \nabla \n \nabla \n + (\nabla \cdot \n) \nabla \n \Bigr) \right) = s (\nabla \cdot \n) \mathcal{ST} (\nabla \n) + s ~ \mathcal{ST} \left( (\nabla^2 \n) \n + \nabla \n \nabla \n  \right), \\
& \mathcal{ST} \Bigl( s ~ \n \otimes \nabla(\nabla \cdot \n) \Bigr) 
= \left( s \nabla(\nabla \cdot \n) \cdot \n \right) \left( \n \odot \n - \frac{1}{3} \mathbf{I}   \right) + \n \odot \Bigl( s \left( \nabla(\nabla \cdot \n) - \left( \nabla(\nabla \cdot \n) \cdot \n \right) \n \right)  \Bigr). \\
\end{aligned}
\end{equation}

Hence, we have
\begin{equation}
  \begin{aligned}
    & \frac{1}{2} \Bigl( \Qvec_{ik,kj}  + \Qvec_{jk,ki} - \frac{2}{3} \delta_{ij} \Qvec_{kl, kl} \Bigr) = \mathcal{ST} \left( \Qvec_{ik, kj} \right) \\
    & =  \left( (\nabla \cdot \n) (\nabla s \cdot \n) + \nabla s \cdot (\nabla \n) \n + s \nabla (\nabla \cdot \n) \cdot \n + (\nabla^2 s) \n \cdot \n \right) \left(\n \odot  \n - \frac{1}{3} \mathbf{I} \right)\\
    & + \n \odot \Bigl( \left( (\nabla^2 s) \n - ((\nabla^2 s) \n \cdot \n) \n \right) + \left( (\nabla \n)^{\mathrm{T}} \nabla s - (\nabla s \cdot (\nabla \n) \n) \n \right) + (\nabla s \cdot \n)  (\nabla \n) \n  \\
    & ~~~~~~~~~~~ + (\nabla \cdot \n) \left( \nabla s - (\nabla s \cdot \n) \n \right) + s \left( \nabla (\nabla \cdot \n) -  \left( \nabla (\nabla \cdot \n) \cdot \n \right) \n \right) \Bigr) + R(s, \n), \\
  \end{aligned}
\end{equation}
 where
\begin{equation}
  \begin{aligned}
 R(s, \n) & = \mathcal{ST}  \bigl( (\nabla \n) \n \odot (\nabla s - (\nabla s \cdot \n)  \n ) \bigr)  + (\nabla s \cdot \n + s \nabla \cdot \n) \mathcal{ST}(\nabla \n)  \\
      & + s~ \mathcal{ST} \bigl( \nabla \n \nabla \n + (\nabla^2 \n)\n \bigr) - \frac{1}{3} \left( \nabla^2 s - \frac{1}{3} (\Delta s) \mathbf{I} \right). \\
  \end{aligned}
\end{equation}

Note
\begin{equation}
\begin{aligned}
J(s, \n) & \triangleq   \left( (\nabla^2 s) \n - ((\nabla^2 s) \n \cdot \n) \n \right) + \left( (\nabla \n)^{\mathrm{T}} \nabla s - (\nabla s \cdot (\nabla \n) \n) \n \right) + (\nabla s \cdot \n)  (\nabla \n) \n  \\
    & + (\nabla \cdot \n) \left( \nabla s - (\nabla s \cdot \n) \n \right) + s \left( \nabla (\nabla \cdot \n) -  \left( \nabla (\nabla \cdot \n) \cdot \n \right) \n \right) \in \n^{\perp}, \\
\end{aligned}
\end{equation}
so $\n \odot J(s, \n)  \in V_2$.

\end{appendices}

\bibliography{LC}

\end{document}